\documentclass{amsart}

\usepackage{graphicx}
\usepackage{tikz}
\usepackage{tikz-cd}
\usetikzlibrary{graphs,graphs.standard,quotes}

\usetikzlibrary{snakes}
\usepackage[all]{xy}
\usepackage{framed}
\usepackage{mathrsfs}
\usepackage{amssymb}
\usepackage{blkarray, bigstrut} 
\usepackage{colonequals} 
\usepackage{hyperref}
\usepackage{url} 

\theoremstyle{plain}
\newtheorem{theorem}{Theorem}[section]
\newtheorem{lemma}[theorem]{Lemma}
\newtheorem{corollary}[theorem]{Corollary}
\newtheorem{proposition}[theorem]{Proposition}

\theoremstyle{definition}
\newtheorem{definition}[theorem]{Definition}
\newtheorem{example}[theorem]{Example}

\theoremstyle{remark}
\newtheorem{remark}[theorem]{Remark}

\numberwithin{equation}{section}

\newcommand{\CC}{\mathbb{C}}

\newcommand{\QQ}{\mathbb{Q}}
\newcommand{\RR}{\mathbb{R}}
\newcommand{\ZZ}{\mathbb{Z}}
\newcommand{\NN}{\mathbb{N}}
\newcommand{\Hom}{{\rm Hom}}

\newcommand{\ep}{\epsilon}

\newcommand{\mc}{\mathcal}
\newcommand{\ms}{\mathscr}
\newcommand{\mbb}{\mathbb}

\newcommand{\lra}{\longrightarrow}
\newcommand{\letbe}{\colonequals}
\newcommand{\largewedge}{\mbox{\Large $\wedge$}}

\begin{document}

\title{Equivariant cohomology of torus orbifolds}

\author{Alastair Darby}
\address{Department of Mathematical Sciences, Xi'an Jiaotong-Liverpool University, 111 Ren'ai Road,
Suzhou, 215123, Jiangsu Province, China}
\curraddr{}
\email{Alastair.Darby@xjtlu.edu.cn}

\author{Shintaro Kuroki}
\address{Department of Applied Mathematics, Okayama University of Science, 1-1 Ridai-cho Kita-ku Okayama-shi Okayama 700-0005, Japan}
\curraddr{}
\email{kuroki@xmath.ous.ac.jp}

\author{Jongbaek Song}
\address{School of Mathematics, KIAS, 85 Hoegiro Dongdaemun-gu, Seoul 02455, Republic of Korea}
\curraddr{}
\email{jongbaek.song@gmail.com}


%
%



\begin{abstract}
We calculate the integral equivariant cohomology, in terms of generators and relations, of locally standard torus orbifolds whose odd degree ordinary cohomology vanishes. We begin by studying GKM-orbifolds, which are more general, before specialising to half-dimensional torus actions.
\end{abstract}

\maketitle

\section{Introduction}
The interest in the integral cohomology of orbifolds stems from the subtleties that appear when in comparison with their manifold counterparts. The main, and first, example of this is \emph{weighted projective space}. It can easily be seen that its rational cohomology is the same as that of ordinary projective space. Kawasaki~\cite{Ka} proved, surprisingly, that integrally their cohomologies are also additively isomorphic but that they have different product structures. Weighted projective spaces, when thought of as toric varieties, admit a natural half-dimensional compact torus action whose equivariant cohomology was studied in~\cite{BFR}. A similar phenomenon is observed whereby the rational equivariant cohomology is the same as for ordinary projective space but the integral equivariant cohomology distinguishes them. 
In this paper we consider a much wider class of orbifolds with torus action where this phenomenon persists, namely GKM-orbifolds and torus orbifolds, and calculate their integral equivariant cohomology.

Goresky, Kottwitz and MacPherson~\cite{GKM} showed that a lot could be said about the equivariant cohomology of a wide range of spaces with compact torus action by considering a combinatorial approximation of the space. More specifically, if a space $X$ with compact torus $T^k$-action is \emph{equivariantly formal} (which is implied if $H^{odd}(X)=0$), then each closed one-dimensional orbit is a copy of a 2-sphere that is rotated according to some element of $\Hom(T^k,S^1)$. They proved that the equivariant cohomology $H^*_{T^k}(X;\RR)$ (note the real coefficients) is encoded in the 1-skeleton, the union of the zero- and one-dimensional orbits. It is isomorphic to the algebra of \emph{piecewise polynomials} on the 1-skeleton, i.e.\ the algebra given by attaching to each fixed point of $X$ elements of the polynomial algebra $H^*(BT^k;\RR)$ so that if two fixed points belong to a 2-sphere their polynomials agree modulo the element of $\Hom(T^k,S^1)$ attached to that sphere.

The 1-skeleton can be modeled on a \emph{labelled graph} and Guillemin and Zara~\cite{GZ} abstract this idea, studying these graphs as objects of interest in their own right, and coin the terms GKM-manifold, GKM-orbifold and GKM-graph. The algebra of piecewise polynomials on the 1-skeleton described beforehand is extended to the abstract labelled graphs and is often called the \emph{equivariant cohomology of the graph}. A similar technique can be used for the equivariant cohomology of GKM-manifolds with integer coefficients where the polynomials used now lie in $H^*(BT^k;\ZZ)$. 

Torus manifolds, first studied in~\cite{HM} and consequently in~\cite{MP,MMP}, can be considered as a special case of GKM-manifolds where the the torus acts effectively with exactly half the dimension of the manifold. This has the added bonus that we can explicitly give generators and relations for the integral equivariant cohomology. Examples of torus manifolds include toric manifolds (defined as compact non-singular toric varieties), quasitoric manifolds and even-dimensional spheres.

The integral equivariant cohomology of smooth toric varieties is known to be given by the \emph{face ring} of the corresponding fan and for quasitoric manifolds the face ring of the quotient simple polytope. For the wider class of torus manifolds it was proved by Masuda and Panov~\cite{MP} that the integral equivariant cohomology is isomorphic to the face ring of an appropriate simplicial poset. In~\cite{MMP} they use \emph{Thom classes} of the associated \emph{torus graph} to explicitly give the generators and relations for the integral equivariant cohomology of the graph. This is isomorphic to the integral equivariant cohomology of the torus manifold if its ordinary cohomology vanishes in all odd degrees.

When we move from manifolds to orbifolds the picture slightly changes. In the case of singular toric varieties, including toric varieties having orbifold singularities, it was proved in~\cite{BFR,Fr} that its equivariant cohomology with integer coefficients is given by the ring $PP[\Sigma]$ of piecewise polynomials on its fan $\Sigma$ if its ordinary odd degree cohomology vanishes. 
Note in particular that the ordinary cohomology of a toric manifold is concentrated in even degrees, hence its equivariant cohomology is isomorphic to $PP[\Sigma]$, which is isomorphic to the face ring of the fan in this case. However,  this is not true for orbifolds in general. 
In \cite{BSS, BNSS, KMZ}, several verifiable conditions for toric orbifolds to have vanishing odd degree cohomology are studied, and the authors of \cite{BSS}  show that, under the condition of vanishing odd degree cohomology, the equivariant cohomology of a projective toric orbifold can be realised as a subring of the usual face ring of the fan that satisfies an \emph{integrality condition}. 

GKM-orbifolds, as first defined in~\cite{GZ}, are orbifold versions of GKM-manifolds. The closed one-dimensional orbits are now \emph{spindles} (see Example~\ref{ex_spindle}). Examples of GKM-orbifolds include weighted Grassmanians~\cite{AM, CR-wgr} and weighted flag varieties~\cite{IB-wfl}. In Section~2 we define an abstract orbifold GKM-graph, similar to~\cite{GZ}, which is in some sense a rational version of the manifold case. After showing how a GKM-orbifold produces such a graph we then define the integral cohomology of an orbifold GKM-graph $H^*_T(\Gamma,\alpha)$ and prove that the integral equivariant cohomology of the GKM-orbifold is isomorphic to $H^\ast_T(\Gamma,\alpha)$ when its odd degree cohomology vanishes (amongst other conditions).

In Section~3, as in the manifold case, we concentrate on half-dimensional torus actions. We work entirely combinatorially in this section by initially considering the orbifold analogue of a torus graph. We then define the \emph{weighted face ring} of an orbifold torus graph, which is a polynomial ring given in terms of generators and relations using integral linear combinations of \emph{rational Thom classes}, and show that this weighted face ring is isomorphic to the equivariant cohomology of the orbifold torus graph. 

We move back towards geometry in Section~4, where we consider \emph{locally standard} GKM-orbifolds. This leads to objects known as \emph{torus orbifolds} (see~\cite{GKRW,KMZ}), orbifolds with a half-dimensional torus action, whose quotient space $Q$ is a \emph{manifold with faces}. Each such orbifold can be characterised by labelling the codimension-one faces of $Q$ which is \emph{dual} to the the corresponding torus graph. From this characterisation there is a canonical method to build a torus orbifold using the \emph{quotient construction}, which we show reproduces the original torus orbifold. Using results from the previous sections, we give the integral equivariant cohomology of this torus orbifold when its odd degree cohomology vanishes. 

In Section~5 we give an explicit formula for the integral equivariant cohomology of all 4-dimensional torus orbifolds for which there is a complete obstruction to the vanishing of the odd degree cohomology.

Throughout this paper, cohomology is taken with integral coefficients, unless stated otherwise. As is natural, we make the usual identification $\Hom(S^1,T^k)=H_2(BT^k)\cong \ZZ^k$ and write its standard integral basis as $\{ \ep_1,\dots,\ep_k\}$, where $\ep_i$ is the inclusion into the $i$-th copy of $S^1$ in $T^k\letbe S^1\times\dots\times S^1$. The dual of this notion gives us the identification $\Hom(T^k,S^1)=H^2(BT^k)\cong \ZZ^k$ with its dual integral basis $\{ \ep_1^*,\dots,\ep_k^*\}$, where the elements now relate to projections.

One other way to consider the identification $\Hom(T^k,S^1)=H^2(BT^k)$ is to consider elements of $\Hom(T^k,S^1)$ as complex one-dimensional $T^k$-representations. These can be thought of as $T^k$-equivariant complex line bundles over a point and taking the equivariant first Chern classes of these bundles produces the identification. As $T^k$ is abelian, the one-dimensional complex representations form all of the irreducible representations.

\subsubsection*{Acknowledgements}
The third author would like to thank Matthias Franz for his valuable comments. He is also grateful to Tony Bahri and Soumen Sarkar for their helpful suggestions and encouragement.

The first author was supported by XJTLU RDF-17-02-50.
The second author was supported by JSPS KAKENHI 
Grant Number 17K14196. 
The third author has been supported by Basic Science Research Program through the National Research Foundation of Korea (NRF) funded by the Ministry of Education (NRF-2018R1D1A1B07048480) and a KIAS Individual Grant (MG076101) at Korea Institute for Advanced Study. 

\section{GKM-orbifolds}
We begin with a brief introduction of torus actions on orbifolds. We refer to~\cite[Section~2]{LT} and~\cite[Section~2]{GKRW} for more details regarding Lie group actions on orbifolds.

Let $\mathcal{X}=(X, \mathcal{U})$  be a $2n$-dimensional orbifold, 
where $\mathcal{U}$ denotes an orbifold atlas on the underlying topological 
space $X$. To be more precise, 
$\mathcal{U}$ consists of the maximal collection of orbifold charts 
\begin{equation*}\label{eq_orbifold_chart}
\left\{ (\tilde{U}, G, \phi \colon \tilde{U} \longrightarrow U \subseteq X)\right\}
\end{equation*}
which covers $X$, where $\tilde{U}$ is an open subset 
of $\RR^{2n}$, $G$ is a finite subgroup of $O(2n)$ and $\phi$ is a 
$G$-equivariant map which induces 
a homeomorphism from $\tilde{U}/G$ to an open subset $U$ of $X$.

Suppose that there is an action of a $k$-dimensional compact 
torus $T^k$ on $X$, for some $k\leq n$, with nonempty fixed point set. 
Then we may consider a $T^k$ invariant neighbourhood $U_p$ 
of a fixed point $p$ in $X$. Each 
element $t\in T^k$ 
induces a smooth map $L_t\colon U_p \to U_p$, which yields a smooth lift 
$\tilde{L}_t \colon \tilde{U}_p \to  \tilde{U}_p$ such that the following diagram 
commutes:
\begin{equation*}\label{eq_lifting_action_to_upstairs}
\begin{tikzcd}
\tilde{U}_p \arrow{r}{\tilde{L}_t}\arrow{d}{/G_p}  & \tilde{U}_p \arrow{d}{/G_p} \\
U_p \arrow{r}{L_t} &  U_p .
\end{tikzcd}
\end{equation*}
Notice that $(\tilde{U}_p, G_p, \tilde{U}_p \xrightarrow{/G_p} U_p)$ 
forms an orbifold chart, and $G_p$ is called the \emph{local group} around $p$.  

Now we consider the tangent space $T_p\tilde{U}(\letbe T_p \tilde{U}_p)$ 
of $\tilde{U}_p$ at $p$ and the tangential 
representation with respect to the action of $T^k$ on $\tilde{U}_p$. 
Since $T^k$ is abelian, the tangent space $T_p \tilde{U}$ is decomposed into complex 
$1$-dimensional representations
\begin{equation}\label{eq_tangential_repn}
T_p\tilde{U}\,\cong\,\bigoplus_{i=1}^n V(\alpha_{i, p}),
\end{equation}
where $V(\alpha_{i,p})$ denotes a complex $1$-dimensional 
$T^k$-representation with weight 
$\alpha_{i,p}\in {\rm Hom} (T^k, S^1)$.

\begin{definition}\label{def_GKM_orb}
A closed $2n$-dimensional oriented orbifold 
$\mathcal{X}=(X, \mathcal{U})$ with an action of a
$k$-dimensional torus $T^k$, with $k\leq n$, is called 
a \emph{GKM-orbifold} if 
\begin{enumerate}
\item there are finitely many $T^k$-fixed 
points and connected components of 1-dimensional orbits;
\item the weights $\{\alpha_{1,p}, \dots, \alpha_{n, p} \} \subset  \Hom(T^k, S^1)$ 
of the tangential representation at each fixed point $p\in X$, 
as in \eqref{eq_tangential_repn}, are pairwise linearly independent.
\end{enumerate}
\end{definition}

\begin{remark}
Though Guillemin and Zara~\cite{GZ} require a $T^k$-invariant almost 
complex structure for their GKM-manifolds and GKM-orbifolds, we 
shall not assume the existence of an invariant almost complex structure. 
Also see Remark \ref{rmk_almost_complex_str}. 
\end{remark}

\begin{example}[Spindle]\label{ex_spindle}
Consider the quotient space $$S^{2}(m,n)\,\letbe\,S^3/S^1\langle m,n\rangle,$$ 
where we identity two points $(z_1, z_2)$ and $(z_1', z_2')$ in $S^3\subset \CC^2$ if 
\begin{equation}\label{eq_equiv_class_spindle}
(z_1', z_2')\,=\,(t^mz_1, t^nz_2)
\end{equation}
for some positive integers $m,n$ and $t\in S^1$. We denote by $[z_1:z_2]$ the equivalence class coming from~\eqref{eq_equiv_class_spindle}. 
Though $S^{2}(m,n)$ is equipped with an orbifold structure
induced from the equivalence relation, its underlying topological space 
is homeomorphic to $S^2$. Hence, we may think of the $S^1$-action on 
$S^2(m,n)$ being induced from the standard $T^2/S^1\langle m,n\rangle\cong S^1$-action on $S^2$, 
whose fixed point set consists of two isolated points that are connected by 
a connected component of $1$-dimensional orbits. We refer to 
\cite[Example 1.2.1]{GZ} and \cite{DKS-proc}. 
\end{example}

The two assumptions in Definition \ref{def_GKM_orb} 
lead us to understand the set of 
$1$-dimensional orbits as a union of $2$-spheres which are connected only
by fixed points. This brings about a graph by associating the set of fixed points 
and the set of $1$-dimensional orbits with the set of vertices 
and the set of edges, respectively. 

Furthermore, the tangential 
representation \eqref{eq_tangential_repn} together with 
the order of the local group around a fixed point 
produces a certain labelling on each edge. To be more precise, the orientability of $\mc{X}$ allows the local group $G_p$ to be taken as a finite subgroup of 
$SO(2n)$, and by~\cite[Proposition~2.8]{GKRW}, there exists a Lie group $\tilde{T}$ such that $\tilde{T}$ acts on $\tilde{U}_p$ and $\tilde{T}$ is an extension of $T^k$ by $G_p$. We can also say that $G_p$ commutes with every connected subgroup of $\tilde{T}$ using~\cite[Corollary~2.9]{GKRW}.  Since the identity component $\tilde{T}^o$ (which as a connected, compact, abelian Lie group is therefore a torus) of $\tilde{T}$ has a pairwise linear independent representation  $\CC^{n}\simeq \bigoplus_{i=1}^{n} V(\epsilon_{i}^{*})$, any two-dimensional complex representation $\epsilon_{i}^{*}\times \epsilon_{j}^{*}:\tilde{T}^o\to U(2)$ does not commute with the elements of U(2) except in the case of a maximal torus $T^2\simeq {\rm im} (\epsilon_{i}^{*}\times \epsilon_{j}^{*})$.
This implies that the centraliser of $\tilde{T}^o$ in $SO(2n)$ coincides with a maximal torus and that $G_p$ must be a subgroup of this. Therefore, $G_{i,p}$, the projection of $G_p$ onto its $i$-th coordinate, is a subgroup of $SO(2)$. Here, we notice that the order $|G_p|$ of $G_p$ is a multiple of $|G_{i,p}|$, for each $i=1, \dots, n$.


To get a labelling on each edge, which encodes the given orbifold structure, 
we consider the following composition:
\begin{equation}\label{eq_orb_HOM}
\begin{tikzcd}
T^k \arrow{r}{\alpha_{i,p}} & S^1\arrow{r}{\xi_{i,p}} &S^1/G_{i,p},
\end{tikzcd}
\end{equation}
where $\alpha_{i,p}\in \Hom(T^k, S^1)$ and $\xi_{i,p}(e^{2\pi i x}) = [e^{2\pi i x}]$.
Identifying $\Hom(T^k, S^1)$ with $H^2(BT^k)$, 
the composition $\xi_{i,p}\circ \alpha_{i,p}$ in \eqref{eq_orb_HOM} defines 
an element in $H^2(BT^k; \QQ)$.

Motivated by this geometric interpretation, we 
give the following definition 
of an \emph{abstract orbifold GKM-graph}. 

\begin{definition}\label{def_orb_GKM_graph}
For two positive integers $k$ and $n$, with $k\leq n$, 
an \emph{orbifold GKM-graph} 
is a triple $(\Gamma, \alpha, \theta)$ defined as follows:
\begin{enumerate}
\item $\Gamma$ is an $n$-valent graph with the set 
$\mathcal{V}(\Gamma)$ of vertices and the set 
$\mathcal{E}(\Gamma)$ of oriented edges. 
\item $\alpha \colon \mathcal{E}(\Gamma) \to H^2(BT^k;\QQ)$ is a map
such that 
\begin{enumerate}
\item the set of vectors $\alpha(\mathcal{E}_p(\Gamma))$ are 
pairwise linearly independent for every $p\in \mathcal{V}(\Gamma)$, 
where $\mathcal{E}_p(\Gamma)$ is the set of outgoing edges from $p$;
\item for an oriented edge $e\in \mathcal{E}(\Gamma)$, 
there are positive integers 
$r_{e}$ and $r_{\overline{e}}$ such that 
$r_{e}\alpha(e)=\pm r_{\overline{e}}\alpha(\overline{e})
\in H^2(BT^k)$, where $\overline{e}$ denotes the edge 
$e$ with the reversed orientation. 
\end{enumerate}
\item $\theta$ is a collection of bijections 
$$\theta_e \colon \mathcal{E}_{i(e)}(\Gamma) \longrightarrow \mathcal{E}_{t(e)}(\Gamma),$$
such that 
$c_{e,e'}(\alpha(\theta_{e}(e'))-\alpha(e')) 
\equiv 0\mod r_e\alpha(e)(=\pm r_{\overline{e}}\alpha(\overline{e}))  
\in H^2(BT^k)$, for some $c_{e,e'} \in \ZZ\setminus \{0\}$, where $i(e)$ is the initial vertex and $t(e)$ is the terminal vertex of $e\in \mathcal{E}(\Gamma)$.
\end{enumerate}
The function $\alpha$ and the collection $\theta$ are called an \emph{axial function} 
and a \emph{connection} on $\Gamma$, respectively. 
\end{definition}


Given a GKM-orbifold $\mc{X}$ we immediately obtain a labelled graph $(\Gamma,\alpha)$. We can choose a  connection $\theta$ on $\Gamma,\alpha$ in a similar way to~\cite{GZ}: let $e\in \mathcal{E}(\Gamma)$, with $p=i(e)$ and $q=t(e)$, and $\mathcal{E}_p(\Gamma)=\{ e_1^p=e,e_2^p,\dots,e_n^p\}$ and $\mathcal{E}_q(\Gamma)=\{ e_1^q=\bar{e},e_2^q,\dots,e_n^q\}$. Then the restriction to $S_e$ (the spindle associated to $e$) of the tangent orbibundle to $\mc{X}$ splits equivariantly as a sum of line orbibundles $\bigoplus_{i=1}^n L_i$ and we can relabel the elements in $\mathcal{E}_p(\Gamma)$ and $\mathcal{E}_q(\Gamma)$ so that $L_i|_p=T_pX_{e_i^p}$ and $L_i|_q=T_qX_{e_i^q}$. From this we get the identification $e_i^p\mapsto e_i^q$ which defines a connection on $\Gamma,\alpha$ and produces an orbifold GKM-graph. The integer $c_{e,e_i^p}$ can then be seen to be the Chern number of the vector bundle $\varphi^*L_i$ where $\varphi\colon \CC P^1\to S_e=S^2(m,n)$ (for some $m,n\in \ZZ\setminus 0$) is defined by $\varphi [z_1:z_2]=[z_1^m:z_2^n]$.

\begin{remark}
If we choose $r_e$ and $r_{\overline{e}}$ as the minimal integers such that $r_{e}\alpha(e)=\pm r_{\overline{e}}\alpha(\overline{e})$, then
a multiple of $r_e$ coincides with the order of the local group $G_p$ 
of an orbifold chart around $p=i(e)$. 
In particular, if a fixed point $p\in X^T \cong \mathcal{V}(\Gamma)$ 
is a smooth point, i.e.\ an orbifold chart 
around $p$ has the trivial local group, then we may choose  
$r_{e}$ to be $+1$ 
for each $e\in \mathcal{E}_p(\Gamma)$. 
\end{remark}
\begin{remark}\label{rmk_almost_complex_str}
An \emph{almost complex structure} on $\mathcal{X}$
is an endomorphism $J\colon T\mathcal{X} \to T\mathcal{X}$ 
satisfying $J^2=-id$. If a GKM-orbifold is equipped with an 
$T^k$-invariant almost complex structure $J$, then the orientation 
induced from $J$ forces the associated axial function to
satisfy $r_{e}\alpha(e)=- r_{\overline{e}}\alpha(\overline{e})
\in H^2(BT^k)$, as in  \cite[Definition 2.1.1]{GZ}.
\end{remark}


\begin{example}[Spindles] Here, we calculate two orbifold GKM-graphs $(\Gamma, \alpha, \theta)$ for a spindle $S^2(m,n)$ with respect to two different $S^1$-actions on it. 
\begin{enumerate}
\item If we suppose that the pair $(m,n)$ are coprime integers, then there exist $a,b\in \ZZ$ such that $mb-na=1$. We can then write the $S^1$-action on the spindle $S^2(m,n)$ in Example~\ref{ex_spindle} as
\[
t\cdot[z_1: z_2]\,=\,[t^az_1:t^bz_2],
\]
for a choice of $a$ and $b$ as above, since $t\mapsto [t^a :t^b]$ defines an isomorphism $S^1\to T^2/S^1\langle m,n\rangle$. Note that this action is effective and has two fixed points $[1:0]$ and $[0:1]$. The graph $\Gamma$ is given by the two vertices connected by an edge, and the connection $\theta$ is the only one. To obtain the axial function $\alpha$, we see that
\[
t\cdot [1:z_2]\,=\, [t^a:t^bz_2]\,=\, [(t^{-\frac{a}{m}})^{m}t^{a}: (t^{-\frac{a}{m}})^{n}t^{b}z_{2}]\,=\, [1:t^{\frac{-an+bm}{m}}z_{2}]\,=\,  [1:t^{\frac{1}{m}}z_2],
\]
where the second equality holds because of the relation \eqref{eq_equiv_class_spindle}.
Therefore the weight of the orbifold tangential representation $T_{[1:0]}U=V(\alpha_1)/\ZZ_m$ around 
the fixed point $[1:0]$ is given by $\frac{1}{m}\epsilon^\ast \in H^2(BS^1;\QQ)$, where we denote by $\epsilon^\ast$ the standard basis of $H^2(BS^1)$. Hence, we have $\alpha(e)=\frac{1}{m} \epsilon^\ast$ for the edge $e$ emanating from the vertex corresponding to $[1:0]$. A similar computation for the other fixed point $[0:1]$ induces an orbifold GKM-graph for the $S^1$-action on a spindle $S^2(m,n)$ as described by Figure~\ref{fig_GKM-graph_eff_spindle}.
\item An $S^1$-action on a spindle $S^2(m,n)$ could be given by a diagonal action
$$t\cdot [z_1: z_2]\,=\,[tz_1:  tz_2],$$
where now the action may not be effective. As in the previous example, $\Gamma$ is given by the two vertices connected by an edge and $\theta$ is uniquely determined. We now calculate the weight of the orbifold tangential representation around the fixed point $[1:0]$: 
$$t\cdot [1: z_2]\,=\,[t:  tz_2]\,=\,[1: t^{\frac{m-n}{m}}z_2] $$
and in a similar fashion for the fixed point $[0:1]$. Figure \ref{fig_GKM-graph_spindle} now gives the orbifold GKM-graph for this action.
\end{enumerate}
\end{example}
\begin{figure}
\begin{tikzpicture}
\draw (0,0)--(4,0);
\draw[fill] (0,0) circle (0.05); 
\draw[fill] (4,0) circle (0.05); 
\draw [thick, -> ] (0,0)--(1,0);
\draw [thick, -> ] (4,0)--(3,0);
\node[above] at (1,0.1) {$\frac{1}{m}\epsilon^\ast$};
\node[above] at (3,0.1) {$-\frac{1}{n}\epsilon^\ast$};
\end{tikzpicture}
\caption{The GKM-graph for an effective spindle.}
\label{fig_GKM-graph_eff_spindle}
\end{figure}
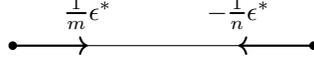
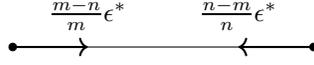
\begin{figure}
\begin{tikzpicture}
\draw (0,0)--(4,0);
\draw[fill] (0,0) circle (0.05); 
\draw[fill] (4,0) circle (0.05); 
\draw [thick, -> ] (0,0)--(1,0);
\draw [thick, -> ] (4,0)--(3,0);
\node[above] at (1,0.1) {$\frac{m-n}{m}\epsilon^\ast$};
\node[above] at (3,0.1) {$\frac{n-m}{n}\epsilon^\ast$};
\end{tikzpicture}
\caption{The GKM-graph for a spindle with diagonal action.}
\label{fig_GKM-graph_spindle}
\end{figure}

\begin{remark}
For more on the spindles $S^2(m,n)$, including describing effective actions when $(m,n)$ are not a pair of coprime integers, see~\cite{DKS-proc}.
\end{remark}

The next example exhibits a $2n$-dimensional GKM-orbifold 
equipped with a $k$-dimensional torus action where $k< n$. 

\begin{example}[Weighted projective space]\label{ex_wCP_111222}
Consider the action of $\CC^\ast$ on $\CC^4$ given by 
$$t\cdot (z_1, z_2, z_3, z_4)=(z_1, tz_2, tz_3, tz_4)$$
for some $t\in \CC^\ast$ and $(z_1, z_2, z_3, z_4)\in \CC^4$. 
It induces an action of $\CC^\ast$ on $\largewedge^2\CC^4 $
as follows
\begin{equation}\label{eq_C*-action_on_wedgeC^4}
t\cdot (z_{12}, z_{13}, z_{14}, z_{23}, z_{24}, z_{34})=
(tz_{12}, tz_{13}, tz_{14}, t^2z_{23}, t^2z_{24}, t^2z_{34}),
\end{equation}
where we write $z_{ij}\letbe z_{i}\wedge z_j$ for simplicity. 
The action defined in \eqref{eq_C*-action_on_wedgeC^4} 
gives us a weighted projective space $\mbb{P}_{(1,1,1,2,2,2)}(\largewedge^2 \CC^4)$. 
Notice that the standard $T^4$-action on $\CC^4$ induces a 
$T^4$-action on $\largewedge^2\CC^4$. Moreover, the circle subgroup
$\{(1, t, t, t)\mid t\in S^1\}$ of $T^4$ acts trivially on $\mbb{P}_{(1,1,1,2,2,2)}(\largewedge^2 \CC^4)$. 
Now, it is straightforward to 
see that $\mbb{P}_{(1,1,1,2,2,2)}(\largewedge^2 \CC^4)$ is a GKM-orbifold
with respect to the residual $T^4/S^1$-action. 
Here, we consider 
the following short exact sequence to identify $T^4/S^1$ with $T^3$:
\begin{equation}\label{eq_ses_for_wCP5}
\begin{tikzcd}
1\arrow{r} 	& S^1 \arrow{r}{\varpi}	&T^4\arrow{r}{\lambda}	& T^3\arrow{r}	&1,
\end{tikzcd}
\end{equation}
where $\varpi(t)=(1,t,t,t)$ and $\lambda$ is defined appropriately so that 
$\ker(\lambda)={\rm im}(\varpi)$, for instance 
$\lambda(t_1, t_2, t_3, t_4)=(t_1, t_2t_4^{-1}, t_3t_4^{-1})$. 
Choose a right splitting
$$\varrho \colon T^3 \lra T^4$$
of \eqref{eq_ses_for_wCP5} defined by $\varrho(r_1, r_2, r_3)=(r_1, r_2, r_3, 1)$, 
which  leads us to identity $T^3$ with ${\rm coker}(\varpi)$. 
Now, we compute the orbifold tangential representation around a fixed point 
$[0:0:0:0:0:1]$ as follows:
\begin{align*}
(r_1, r_2, r_3)&\,\cdot\,  [z_{12}:z_{13}: z_{14}: z_{23}: z_{24}: 1] \\
&=\, [r_1r_2z_{12}: r_1r_3z_{13}: r_1z_{14}: r_2r_3z_{23}: r_2z_{24}: r_3]\\
&=\, [r_1r_2r_3^{-\frac{1}{2}}z_{12}:r_1r_3^{\frac{1}{2}}z_{13}: r_1r_3^{-\frac{1}{2}}z_{14}: r_2z_{23}: r_2r_3^{-1}z_{24}: 1]
\end{align*}
It  shows us that 
$$T_{[0,0,0,0,0,1]}U \,\cong\, \bigoplus_{i=1}^5 V_{\alpha_{i}}/G_{i},$$
where each $1$-dimensional orbifold vector bundle representation $V_{\alpha_{i}}/G_{i}$
is determined by 
$$\left( \epsilon_1^\ast +\epsilon_2^\ast-\frac{1}{2}\ep_3^\ast, 
\ep_1^\ast+\frac{1}{2}\ep_3^\ast, 
\epsilon_1^\ast -\frac{1}{2}\epsilon_3^\ast,  
\ep_2^\ast, 
\epsilon_2^\ast -\epsilon_3^\ast \right)  \,\in\, 
\bigoplus_{i=1}^5 H^2(BT^3;\QQ) ,$$ 
respectively. A similar computation 
for the other fixed points yields the corresponding orbifold GKM-graph which is a complete
graph with six vertices. Figure \ref{fig_GKM_graph_wCP_wedge} shows a part of 
this orbifold GKM-graph around the vertex corresponding to $[0,0,0,0,0,1]$. 
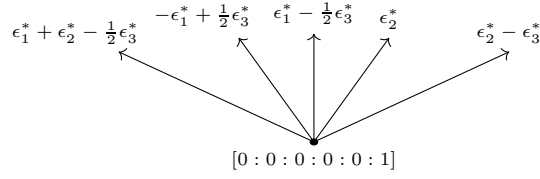
\begin{figure}
\begin{tikzpicture}[yscale=0.8]
\draw[fill] (0,0) circle (0.05); 
\node[below] at (0,0) {\scriptsize$[0:0:0:0:0:1]$};

\foreach \z in {30,150}{
\draw[->] (0:0)--+(\z:3);
}
\foreach \z in {60,120}{
\draw[->] (0:0)--+(\z:2);
}
\foreach \z in {90}{
\draw[->] (0:0)--+(\z:1.8);
}

\node[above] at (155:3.5) {\scriptsize$\epsilon_1^\ast +\epsilon_2^\ast-\frac{1}{2}\ep_3^\ast$};
\node[above] at (130:2.3) {\scriptsize$-\ep_1^\ast+\frac{1}{2}\ep_3^\ast$}; 
\node[above] at (90:1.8) {\scriptsize$\epsilon_1^\ast -\frac{1}{2}\epsilon_3^\ast$};
\node[above] at (60:2) {\scriptsize$\ep_2^\ast$};
\node[above] at (30:3) {\scriptsize$\epsilon_2^\ast -\epsilon_3^\ast$};
\end{tikzpicture}
\caption{Orbifold GKM-graph around $[0:0:0:0:0:1]\in \mbb{P}_{(1,1,1,2,2,2)}(\largewedge^2 \CC^4)$.}
\label{fig_GKM_graph_wCP_wedge}
\end{figure}
\end{example}


Next we introduce an algebraic object obtained from a given orbifold GKM-graph. 
The definition is similar to the one defined in \cite[Section 1.7]{GZ}, but we emphasise that 
we use $\ZZ$ as the coefficient ring. 

\begin{definition}\label{def_equiv_cohom_of_graph}
Given an orbifold GKM-graph $(\Gamma, \alpha,\theta)$, 
we define the \emph{equivariant cohomology $H_{T^k}^\ast(\Gamma, \alpha)$ of the graph
 $(\Gamma, \alpha,\theta)$} as follows:
\begin{equation*}\label{eq_GKM-graph_cohom}
H_{T^k}^\ast(\Gamma, \alpha)\,=\,\left\{f\colon \mathcal{V}(\Gamma) \longrightarrow
H^\ast(BT^k) \mid 
f(i(e))\equiv f(t(e)) \text{ mod } \tilde{r}_e \alpha(e) \right\}, 
\end{equation*}
where $\tilde{r}_e$ is the smallest positive integer satisfying the 
condition (2)-(b) of Definition \ref{def_orb_GKM_graph}. 
\end{definition}
We notice that $H_{T^k}^\ast(\Gamma, \alpha)$ has a natural graded ring structure 
given by vertex-wise addition and multiplication, and the $i$-th degree is defined by  
\[
H_{T^k}^i(\Gamma, \alpha)=\left\{f\colon \mathcal{V}(\Gamma) \longrightarrow
H^i(BT^k) \mid 
f(i(e))\equiv f(t(e)) \text{ mod } \tilde{r}_e \alpha(e) \right\}.
\]
It is easy to check that $H_{T^k}^\ast(\Gamma, \alpha)\colonequals \bigoplus_{i\geq 0} H_{T^k}^i(\Gamma, \alpha)$. 
Moreover, 
it is equipped with a natural $H^\ast(BT^k)$-algebra structure.

\begin{remark}
If $(\Gamma, \alpha, \theta)$ is defined from a smooth toric manifold \cite{DJ}, then the equivariant cohomology $H_{T^k}^\ast(\Gamma, \alpha)$ is isomorphic to the face ring of the orbit space. We will discuss this further in Section \ref{sec_torus_orb_graph} and Section \ref{sec_torus_orb} in a wider category of spaces including all smooth toric manifolds. 
\end{remark}

The equivariant cohomology of $(\Gamma, \alpha, \theta)$ can be defined 
over rational or real coefficients, which coincides with the definition of the 
cohomology ring of a graph in  \cite[Section 1.7]{GZ}:
\begin{equation*}\label{eq_rational_graph_equiv_cohom}
H_{T^k}^\ast(\Gamma, \alpha;\QQ)\,=\,
\left\{f\colon \mathcal{V}(\Gamma) \longrightarrow 
H^\ast(BT^k;\QQ)
 \mid f(i(e))\equiv f(t(e)) \text{ mod } \alpha(e) \right\}.
\end{equation*}
It has a natural 
$H^\ast(BT^k;\QQ)$-algebra structure. We notice that 
$H_{T^k}^\ast(\Gamma,\alpha)$ is a subring of 
$H_{T^k}^\ast(\Gamma, \alpha;\QQ)$.

\begin{theorem}\label{thm_geometric_side}
Let $\mathcal{X}=(X, \mathcal{U})$ be a GKM-orbifold with 
respect to an effective action of torus $T$, whose underlying topological 
space $X$ is homotopic to a $T$-CW complex. 
Assume that all isotropy subgroups of $T$ are connected and $H^{odd}(X)=0$.  
Then the equivariant cohomology ring $H^\ast_{T}(X)$ 
is isomorphic to $H_{T}^\ast(\Gamma, \alpha)$ as an
$H^\ast(BT)$-algebra.
\end{theorem}

\begin{proof}
The proof is similar to  \cite[Proposition 2.2]{BFR} and 
\cite[Theorem 1.3]{Fr}. The assumption $H^{odd}(X;\ZZ)=0$ implies that 
the Serre spectral sequence of the fibration 
$$\begin{tikzcd}
X \arrow{r} & ET\times_T X \arrow{r}& BT\end{tikzcd}$$
degenerates at the $E_2$-page. This is equivalent to the exactness of the Chang--Skjelbred sequence 
\begin{equation}\label{eq_CS_seq}
\begin{tikzcd}
0 \arrow{r}& 
H^\ast_{T^k}(X) \arrow{r}{\iota^\ast}&
H^\ast _{T^k}(X_0)\arrow{r}{\delta} &
H^{\ast+1}_{T^k}(X_1, X_0)\arrow{r}& \cdots,
\end{tikzcd}
\end{equation}
over $\ZZ$-coefficients by the connectedness of the isotropy 
subgroups of $T$. We refer to \cite[Theorem 1.1]{FP}. 
Here, $X_0$ and $X_1$ denote the set of fixed points and the union of all 
$0$- and $1$-dimensional orbits in $X$, respectively. The map $\iota \colon X_0 \hookrightarrow X$
is the inclusion of fixed points into $X$. 
Since $X_0$ is finite, we have 
$$H^\ast_{T^k}(X_0)\,\cong\, 
\bigoplus_{p\in X_0} H^\ast(BT^k).
$$

Let $X_e$ be the connected component of the fixed point set
$X^{T_e}$ corresponding to the edge $e$, where we identify $H^2(BT)$ with $\Hom(T, S^1)$, and 
$T_e\letbe\ker (\tilde{r}_e\alpha(e))$, where $\tilde{r}_e$ is defined as in Definition \ref{def_equiv_cohom_of_graph}. 
Then $X_e$ can be regarded as a set of orbits of $T/T_e\cong S^1$, 
hence it is homeomorphic to a 2-sphere $S^2$. 
Since every circle action on $S^2$
is isomorphic to the standard one, $X_e$ has two fixed points
$\{p,q\}$ which correspond to the vertices of $e$. 
Observe the following commutative diagram:
\begin{equation}\label{eq_proof_of_GKM}
{\begin{tikzcd}[column sep=small]
0 \arrow{r} & 
H^\ast_{T^k}(X_e) \arrow{r}{\iota^\ast_e}& 
H^\ast_{T^k}(\{p, q\}) \arrow{r}{\delta_e} & 
H^{\ast+1}_{T^k}(X_e, \{p, q\}) \arrow{r}& 
\cdots \\
0\arrow{r}&
H^\ast_{T^k}(X_e) \arrow{r}  \arrow[equal]{u} & 
H^\ast_{T^k}(X_e\setminus\{q\}) \oplus H^\ast_{T^k}(X_e\setminus \{p\})\arrow{r}{j_e}   \arrow{u}{\cong}&
H^{\ast}_{T^k}(X_e\setminus \{p,q\})\arrow{r} &
\cdots\\
&
&
H^\ast(BT^k) \oplus H^\ast(BT^k)\arrow{u}{\cong} \arrow{r}{\tilde{j}_e}&
H^\ast(BT_e)\arrow{u}{\cong}&
\end{tikzcd}}
\end{equation}
where the first row is the long exact sequence of the pair $(X_e, \{p, q\})$ and 
the second row is the Mayer--Vietoris sequence, and
the third row is induced from the inclusion $T_e \to T^k$. 
Moreover, $\tilde{r}_e\alpha_e$, as an element of $\Hom(T^k, S^1)$, induces 
a short exact sequence 
$$\begin{tikzcd}
1 \arrow{r} &T_e \arrow{r}&  T^k \arrow{r}{\tilde{r}_e\alpha_e} &S^1 \arrow{r} &1,
\end{tikzcd}
$$
which allows us to identify  $H^\ast(BT_e)$ with $H^\ast(BT^k)/ \left< \tilde{r}_e\alpha(e)\right>$. 

Note that $X_1$ is the union of 2-spheres intersecting only at fixed points.  
Hence, a successive application of the relative Mayer--Vietoris sequence yields
an isomorphism 
$$H_{T^k}^\ast(X_1, X_0) \,\cong\, \bigoplus_{e\in {\mathcal{E}}}H^\ast_{T^k}(X_e, \{p, q\}).$$
Moreover, diagram \eqref{eq_proof_of_GKM} shows that $\ker \delta_e\cong \ker j_e$. 
Hence, by fixing an orientation on $\mc{E}(\Gamma)$, the kernel of the differential $\delta$ in 
\eqref{eq_CS_seq} can be realised as
\begin{align}
\ker \delta \,&\cong\, \bigcap_{e\in \mathcal{E}(\Gamma)} \ker \delta_e \,\cong\, \bigcap_{e\in \mathcal{E}(\Gamma)}\ker j_e \nonumber\\
&=\, \{f \colon \mathcal{V}(\Gamma) \longrightarrow H^\ast(BT^k) \mid 
f(i(e))-f(t(e)) \equiv 0 ~{\rm mod}~(\tilde{r}_e \alpha(e)) \},  \label{eq_ker_delta}
\end{align}
where the equality in \eqref{eq_ker_delta} holds because the map $\tilde{j}_e$ in \eqref{eq_proof_of_GKM} sends $(f(i(e)),f(t(e)))$ to $f(i(e))-f(t(e))$. Notice that the exactness of \eqref{eq_CS_seq} implies that \eqref{eq_ker_delta} is isomorphic to image of monomorphism 
$\iota^\ast \colon H^\ast_{T^k}(X) \to H_{T^k}^\ast (X_0)$. Hence, the result follows. 
\end{proof}

\section{Orbifold torus graphs}\label{sec_torus_orb_graph}
In this section, we focus on the case when $k=n$ from the definition 
of an orbifold GKM-graph and from now on write $T\letbe T^n$. All discussions in this section 
generalise some of the ideas in \cite[Section 3]{MMP} to the orbifold category. 
We begin by defining a particular class of orbifold GKM-graphs as follows: 
\begin{definition}\label{def_orb_torus_graph}
An \emph{abstract orbifold torus graph} is a pair $(\Gamma, \alpha)$ of  
an $n$-valent graph $\Gamma$ and a function  
$$\alpha \colon \mathcal{E}(\Gamma) \longrightarrow H^2(BT;\QQ)$$ 
called an \emph{axial function} such that 
\begin{enumerate}
\item the set of vectors $\alpha(\mathcal{E}_p(\Gamma))$ is linearly 
independent  for every $p\in \mathcal{V}(\Gamma)$;
\item for each edge $e\in \mathcal{E}(\Gamma)$, there are positive integers 
$r_{e}$ and $r_{\overline{e}}$ such that 
$r_{e}\alpha(e)=\pm r_{\overline{e}}\alpha(\overline{e})
\in H^2(BT)$, where $\overline{e}$ denotes the edge $e$ with
the reversed orientation. 
\end{enumerate}
\end{definition}

\begin{remark}\label{rmk_torus_graph}
Though Definition \ref{def_orb_GKM_graph} assumes the pairwise linear independency of $\alpha$, here we further assume the linearly independency of $\alpha$ as in Definition~\eqref{def_orb_torus_graph}-(2). This makes the abstract orbifold torus graph of Definition \ref{def_orb_torus_graph} generalise the notion of a \emph{torus graph} in \cite[Section 3]{MMP} to the orbifold setup. 
\end{remark}  

We notice that the first condition of Definition \ref{def_orb_torus_graph}
determines a connection $\theta$ on $(\Gamma, \alpha)$ uniquely. 
This allows us to define a \emph{face} of an abstract orbifold torus graph as follows: 
let $\Gamma'$ be a $d$-valent subgraph of 
$\Gamma$, where $d< n$. Then
\begin{equation}\label{eq_face_subgraph}
F\,\letbe\, (\Gamma', \alpha|_{\mc{E}(\Gamma')})
\end{equation}
is called a 
$d$-\emph{dimensional face} of $(\Gamma,\alpha, \theta)$ if it is 
invariant under the uniquely determined connection $\theta$. 
Figure \ref{fig_orb_tor_graph_CP1236} is an example of an orbifold torus graph. 
Considering the graph $\Gamma$ as a 1-skeleton of $3$-simplex $\Delta^3$, 
one can see that faces of $(\Gamma, \alpha)$ are given by intersecting 
$(\Gamma, \alpha)$ with faces of $\Delta^3$. 

\begin{figure}
\begin{tikzpicture}[scale=0.8]
\draw (0,0)--(3,5)--(-3,5)--cycle;
\draw (0,10/3)--(3,5);
\draw (0,10/3)--(-3,5);
\draw (0,10/3)--(0,0);
               
\draw [fill] (0,0) circle [radius=0.05];
\draw [fill] (3,5) circle [radius=0.05];	
\draw [fill] (-3,5) circle [radius=0.05];
\draw [fill] (0,10/3) circle [radius=0.05];


\draw[very thick, ->] (0,0)--(0,1);
\draw[very thick, ->] (0,0)--(3/5,1);
\draw[very thick, ->] (0,0)--(-3/5,1);
\node[above] at (0,1) {\scriptsize$-\frac{1}{2}\ep_1^\ast$};
\node[right] at (3/5,0.8) {\scriptsize $-\frac{3}{2}\ep_1^\ast+\ep_2^\ast$};
\node[left] at (-3/5,0.8) {\scriptsize $-3\epsilon_1^\ast+\ep_3^\ast$};

\draw[very thick, ->] (0,10/3)--(0,10/3 -1);
\draw[very thick, ->] (0,10/3)--(1,35/9);
\draw[very thick, ->] (0,10/3)--(-1,35/9);
\node[right] at (0,7/3) {\scriptsize $\ep_1^\ast$};
\node[below] at (1,34/9) 
{\scriptsize $\ep_2^\ast$};
\node[below] at (-1, 34/9) {\scriptsize $\ep_3^\ast$};

\draw[very thick, ->] (3,5)--(2,5);
\draw[very thick, ->] (3,5)--(12/5,4);
\draw[very thick, ->] (3,5)--(2,40/9);
\node[above] at (2,5) {\scriptsize $-2\ep_2^\ast+\ep_3^\ast$};
\node[right] at (13/5,4) {\scriptsize $\ep_1^\ast-\frac{2}{3}\ep_2^\ast$};
\node[left] at (2.1,42/9) {\scriptsize $-\frac{1}{3}\ep_2^\ast$};

\draw[very thick, ->] (-3,5)--(-2,5);
\draw[very thick, ->] (-3,5)--(-2,40/9);
\draw[very thick, ->] (-3,5)--(-12/5,4);
\node[above] at (-2,5) {\scriptsize$\ep_2^\ast-\frac{1}{2}\ep_3^\ast$};
\node[right] at (-2.1,42/9) {\scriptsize$-\frac{1}{6}\ep_3^\ast$};
\node[left] at (-13/5,4) {\scriptsize$\ep_1^\ast-\frac{1}{3}\ep_3^\ast$};
\end{tikzpicture}
\caption{An orbifold torus graph of $(\Gamma, \alpha)$.}
\label{fig_orb_tor_graph_CP1236}
\end{figure}
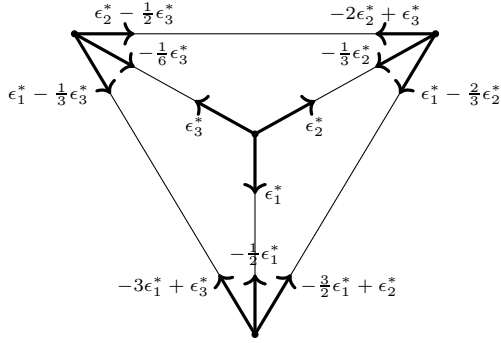

From an orbifold torus graph $(\Gamma, \alpha)$, we shall define an
algebraic object $\ZZ[\Gamma, \alpha]$ which we call a 
\emph{a weighted face ring}. 
Let $\mathcal{F}^{(d)}$ be the set of all $d$-dimensional faces of 
$(\Gamma,\alpha)$ for $0\leq d \leq n$, and 
$\mathcal{F}\letbe \mathcal{F}^{(0)} \cup \dots \cup \mathcal{F}^{(n)}$
the set of all faces of $(\Gamma, \alpha)$.
Notice that 
$\mc{F}^{(n)}=\{\Gamma \}$ , $\mathcal{F}^{(0)}=\mc{V}(\Gamma)$ 
and $\mathcal{F}^{(1)}$ is the set of edges of $\Gamma$ ignoring 
the orientation.

Associating each face $F$ with a formal generator $x_F$, with $\deg x_F=2(n-\dim F)$, 
we obtain a polynomial ring $\mathbf{k}[x_F \mid F\in \mathcal{F}]$ for 
any commutative ring $\mathbf{k}$ with unit. 
We define $x_{\Gamma}=1$ and $x_\emptyset=0$ by convention. In this paper, we mainly 
focus on the case when $\mathbf{k}$ is $\QQ$ or $\ZZ$. 
Now, we consider a ring homomorphism  
\begin{equation}\label{eq_mu:poly->H*BT}
\mu \colon \QQ[x_F \mid F\in \mathcal{F}] \longrightarrow H_T^\ast(\Gamma, \alpha;\QQ)
\end{equation}
sending $x_F$ to the element $\tau_F$ defined by 
\begin{equation}\label{eq_rational_thom_class}
\tau_F(v)\,\letbe\, \begin{cases}
\prod_{\substack{i(e)=v\\e\notin \Gamma'}  }\alpha(e) & \text{if } 
v\in \mc{V}(\Gamma');\\
0 & \text{otherwise}, \end{cases}
\end{equation}
see \eqref{eq_face_subgraph} for the relation between $F$ and $\Gamma'$. 
We call $\tau_F$ the \emph{rational Thom class} corresponding to $F\in \mathcal{F}$. 
See Figure \ref{fig_rat_Thom_classes} for rational Thom class of degree $2$ 
for the graph $(\Gamma, \alpha)$ 
described in Figure \ref{fig_orb_tor_graph_CP1236}. 

\begin{figure}
\begin{tikzpicture}[scale=0.35]
\begin{scope}
\coordinate (v1) at (0,10/3);
\coordinate (v2) at (0,0);
\coordinate (v3) at (3,5);
\coordinate (v4) at (-3,5);

\draw[thick] (0,0)--(3,5)--(-3,5)--cycle;
\draw[dotted] (0,10/3)--(3,5);
\draw[dotted] (0,10/3)--(-3,5);
\draw[dotted] (0,10/3)--(0,0);
               
\draw [fill] (0,0) circle [radius=0.15];
\draw [fill] (3,5) circle [radius=0.15];	
\draw [fill] (-3,5) circle [radius=0.15];

\node[below] at (v2) {\scriptsize$-\frac{1}{2}\epsilon_1^\ast$};
\node[above] at (2.5,5) {\scriptsize $-\frac{1}{3}\ep_2^\ast$};
\node[above] at (-2.5,5) {\scriptsize $-\frac{1}{6}\ep_3^\ast$};

\node at (0,9/3) {\footnotesize$\tau_{F_1}$};         

\end{scope}

\begin{scope}[xshift=250]
\coordinate (v1) at (0,10/3);
\coordinate (v2) at (0,0);
\coordinate (v3) at (3,5);
\coordinate (v4) at (-3,5);

\draw[dotted] (0,0)--(3,5)--(-3,5)--cycle;
\draw[dotted] (0,10/3)--(3,5);
\draw[dotted] (0,10/3)--(-3,5);
\draw[dotted] (0,10/3)--(0,0);

\draw[thick]  (v1)--(v3)--(v4)--cycle;
\draw [fill] (v1) circle [radius=0.15];
\draw [fill] (v3) circle [radius=0.15];
\draw [fill] (v4) circle [radius=0.15];

\node[below] at (v1) {\scriptsize $\ep_1^\ast$};
\node[above] at (2.5,5) {\scriptsize $\ep_1^\ast-\frac{2}{3}\ep_2^\ast$};
\node[above] at (-2.5,5) {\scriptsize $\ep_1^\ast - \frac{1}{3}\epsilon_3^\ast$};

\node at (0,4.3) {\footnotesize$\tau_{F_2}$};

\end{scope}

\begin{scope}[xshift=500]
\coordinate (v1) at (0,10/3);
\coordinate (v2) at (0,0);
\coordinate (v3) at (3,5);
\coordinate (v4) at (-3,5);

\draw[dotted] (0,0)--(3,5)--(-3,5)--cycle;
\draw[dotted] (0,10/3)--(3,5);
\draw[dotted] (0,10/3)--(-3,5);
\draw[dotted] (0,10/3)--(0,0);

\draw[thick] (v1)--(v2)--(v4)--cycle;
\draw [fill] (v1) circle [radius=0.15];
\draw [fill] (v2) circle [radius=0.15];	
\draw [fill] (v4) circle [radius=0.15];

\node[right] at (v1) {\scriptsize $\ep_2^\ast$};
\node[below] at (v2) {\scriptsize $-\frac{3}{2}\ep_1^\ast+\ep_2^\ast$};
\node[above right] at (v4) {\scriptsize $\ep_2^\ast-\frac{1}{2}\ep_3^\ast$};

\node at (-1,3) {\footnotesize$\tau_{F_3}$};

\end{scope}

\begin{scope}[xshift=750]
\coordinate (v1) at (0,10/3);
\coordinate (v2) at (0,0);
\coordinate (v3) at (3,5);
\coordinate (v4) at (-3,5);

\draw[dotted] (0,0)--(3,5)--(-3,5)--cycle;
\draw[dotted] (0,10/3)--(3,5);
\draw[dotted] (0,10/3)--(-3,5);
\draw[dotted] (0,10/3)--(0,0);

\draw[thick]  (v1)--(v2)--(v3)--cycle;
\draw [fill] (v1) circle [radius=0.15];
\draw [fill] (v2) circle [radius=0.15];	
\draw [fill] (v3) circle [radius=0.15];

\node[left] at (v1) {\scriptsize $\ep_3^\ast$};
\node[below] at (v2) {\scriptsize $-3\ep_1^\ast+\ep_3^\ast$};
\node[above left] at (v3) {\scriptsize $-2\ep_2^\ast+\ep_3^\ast$};

\node at (1,3) {\footnotesize$\tau_{F_4}$};

\end{scope}

\end{tikzpicture}
\caption{Rational Thom classes of degree $2$.}
\label{fig_rat_Thom_classes}
\end{figure}
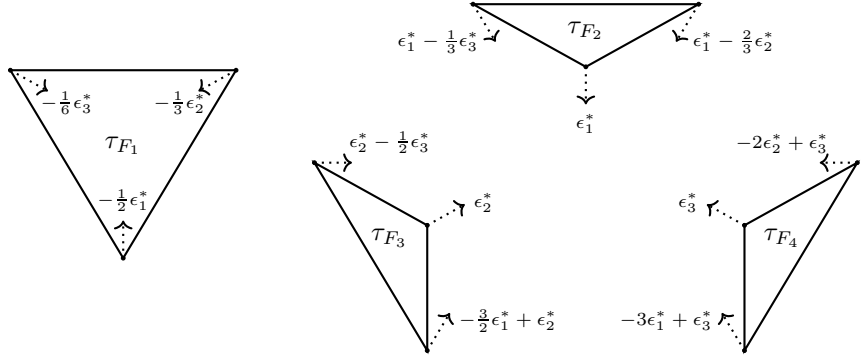

Restricting $\mu$ of (3.1) to the ring $\ZZ[x_F \mid F\in \mathcal{F}]$ of polynomials with
integral coefficients, we have the following subset  
$\ms{Z}_{\Gamma, \alpha}$ of $\ZZ[x_F \mid F\in \mc{F}]$:
\begin{equation}\label{eq_subcollection_integrality_condition}
\mathscr{Z}_{\Gamma, \alpha} \,\letbe\,
\left\{ f\in \ZZ[x_F \mid F\in \mc{F}] ~\big|~ \forall v\in \mc{V}(\Gamma),~
\mu(f)(v) \in  H^\ast(BT) \right\}.
\end{equation}
Indeed, the collection 
\eqref{eq_subcollection_integrality_condition} is 
closed under addition and multiplication induced from 
$\ZZ[x_F \mid F\in \mc{F}]$, hence it is a subring. 
\begin{remark}\label{rmk_w_face_smooth_case}
When $(\Gamma, \alpha)$ is a torus graph as defined in \cite{MMP}, 
the image of the axial function $\alpha$ sits in $H^2(BT)$. Hence, in this case, 
$\ms{Z}_{\Gamma, \alpha}=\ZZ[x_F\mid F\in \mc{F}]$. 
\end{remark}

We notice that the combinatorics of 
$\Gamma$ gives the relation
\begin{equation}\label{eq_SR-ideal_for_Thom_classes}
\tau_F\tau_E\,=\,\tau_{E\vee F}\sum_{G\in E\cap F}\tau_G,
\end{equation}
where $E \vee F$ denotes the minimal face containing both $E$ and $F$ and $G\in E\cap F$ runs through all connected components.
We refer to the proof of \cite[Lemma 6.3]{MP}. 
Therefore, it is straightforward to see from the definition of $\ms{Z}_{\Gamma, \alpha}$ 
that 
\begin{equation}\label{eq_generator_SR_ideal}
\Big\{x_Fx_E-x_{E\vee F}\sum_{G\in E\cap F}x_G \mid E, F, G \in \mathcal{F}  \Big\}
\end{equation}
is a subset of $\ms{Z}_{\Gamma, \alpha}$ and defines an ideal of $\ms{Z}_{\Gamma, \alpha}$
because the image of each element in \eqref{eq_SR-ideal_for_Thom_classes} 
via $\mu$ is identically zero. 
Now we have the following definition 
of a \emph{weighted face ring}. 

\begin{definition}\label{def_weighted_face_ring}
The \emph{weighted face ring} of an abstract torus orbifold graph $(\Gamma, \alpha)$ 
is the quotient ring 
$$\mathbb{Z}[\Gamma, \alpha]\,\letbe\,\ms{Z}_{\Gamma, \alpha}/ \mc{I},$$
where $\mathcal{I}$ is the ideal generated by elements of \eqref{eq_generator_SR_ideal}.
\end{definition}

In the proof of \cite[Theorem 5.5]{MMP}, they consider 
they considered a torus graph $(\Gamma, \alpha)$ and prove that 
the map $\mu$ in \eqref{eq_mu:poly->H*BT}, with $\ZZ$-coefficients, 
factors through the map 
\begin{equation}\label{eq_factor_through_isom}
\ZZ[x_F \mid F\in \mc{F}]/\mc{I} \longrightarrow H^\ast_T(\Gamma, \alpha)
\end{equation}
and it is indeed an isomorphism.  In this case, the domain of \eqref{eq_factor_through_isom} 
coincides with $\ZZ[\Gamma, \alpha]$ of Definition \ref{def_weighted_face_ring}, 
see Remark \ref{rmk_w_face_smooth_case}. 

If we work over rationals, the role of the integers $\tilde{r}_e$ in an abstract orbifold torus graph cancels out, which makes the theory the same as \cite[Theorem 5.5]{MMP}. Hence, we apply the same argument to get the following lemma: 
\begin{lemma}\label{lem_isom_face_ring_graph_cohom_over_Q}
Given a torus orbifold graph $(\Gamma, \alpha)$, 
there is a ring isomorphism between $\QQ[x_F \mid F\in \mc{F}]/\mc{I}$ 
and $H^\ast_T(\Gamma, \alpha; \QQ)$. 
\end{lemma}

The next theorem
extends the result of \cite{MMP} to the category of torus orbifolds. 
\begin{theorem}
\label{main-thm1}
There is a ring  isomorphism between 
$\mathbb{Z}[\Gamma, \alpha]$ and $H_{T}^{*}(\Gamma,\alpha)$. 
\end{theorem}
\begin{proof}
First, we consider a map 
$$\nu \colon \ms{Z}_{\Gamma, \alpha} \longrightarrow H^\ast_T(\Gamma,\alpha),$$ 
which is the restriction of $\mu$ defined in \eqref{eq_mu:poly->H*BT} 
to $ \ms{Z}_{\Gamma, \alpha}$. 
Then we have the
following commutative diagram:
\begin{equation}\label{eq_comm_diag_face_ring_graph_cohom}
\begin{tikzcd}
\QQ[x_F \mid F\in \mc{F}] \arrow[r, two heads, "\mu"]  & 
H^\ast_T(\Gamma, \alpha;\QQ)  \\
\ms{Z}_{\Gamma, \alpha} \arrow[u, tail, "\iota_1"] \arrow[r, "\nu" ] &
H^\ast_T(\Gamma, \alpha). \arrow[u, tail, "\iota_2"]
\end{tikzcd}
\end{equation}
We prove the theorem by showing that $\nu$ is a surjective 
homomorphism with kernel $\mathcal{I}$.

We first show the surjectivity. Lemma \ref{lem_isom_face_ring_graph_cohom_over_Q} implies that $H^\ast_T(\Gamma, \alpha;\QQ)$ is generated by the rational Thom classes $\tau_F$. Hence, any element in $H^\ast_T(\Gamma, \alpha)$ is an integral linear combination of rational Thom classes, say $g\colonequals \sum_{F\in \mathcal{F} }a_F\tau_F$ for some $a_F\in \ZZ$. Moreover, $g$ satisfies 
\begin{equation}\label{eq_cond_for_f(v)}
g(v)\in H^\ast(BT),\quad \text{for all $v\in \mathcal{V}(\Gamma)$,}
\end{equation}
by Definition \ref{def_equiv_cohom_of_graph}, 
which means $\sum_{F\in \mathcal{F} }a_F x_F\in \QQ[x_F \mid F\in \mc{F}]$ is indeed an element in $\mathscr{Z}_{\Gamma, \alpha}$ because of \eqref{eq_subcollection_integrality_condition}. Hence, the map $\nu$ is surjective. 


Next, we show that $\ker \nu=\mathcal{I}$. Recall that the map $\mu$ factors through 
the map 
\[
\QQ[x_F \mid F\in \mc{F}]/\mc{I} \longrightarrow H^\ast_T(\Gamma, \alpha;\QQ)
\]
by Lemma \ref{lem_isom_face_ring_graph_cohom_over_Q}. Hence, we have $\ker\mu =\mathcal{I}$. Moreover, the commutativity of \eqref{eq_comm_diag_face_ring_graph_cohom} gives $\iota_2 \circ \nu = \mu \circ \iota_1$, 
which implies that $\ker\nu\subset \ker \mu=\mathcal{I}$. To show the reverse inclusion, take an element $h\in \ker\mu$. Since $\ker\mu=\mathcal{I}\subset \mathscr{Z}_{\Gamma, \alpha}$, we may assume that $h=\iota_1(h)$. Now, the commutativity of \eqref{eq_comm_diag_face_ring_graph_cohom} together with the injectivity of $\iota_2$ establishes $h\in \ker \nu$. 
\end{proof}

\begin{remark}
For a calculation of the weighted face ring $\ZZ[\Gamma,\alpha]$, when $\Gamma$ has exactly two vertices, see~\cite{DKS-proc}.
\end{remark}

\section{Torus orbifolds}\label{sec_torus_orb}

We now consider \emph{torus orbifolds}, $2n$-dimensional GKM-orbifolds equipped with $n$-dimensional torus actions which may have finite kernels.  
Using the results from the previous sections, we calculate their integral equivariant cohomology ring, in terms of generators and relations, when their ordinary odd degree cohomology vanishes.

\subsection{Locally standard torus orbifolds}
We begin with defining a \emph{locally standard} torus action on a torus orbifold.  
For each point 
$p\in X$, there exists an orbifold chart 
$(\tilde{U}, G, \phi \colon \tilde{U} \to U)$ of a 
$T$-invariant neighborhood $U$ around $p$, 
an equivariant diffeomorphism $\psi$ from an open subset $W$ of $\CC^n$ 
to $\tilde{U}$ and 
a surjective covering homomorphism $\xi \colon (S^1)^n \to T$
with $\ker\xi \cong G$ such that the following diagram commutes:
$$
\begin{tikzcd}
(S^1)^n \times W \arrow[two heads]{rr}{\xi\times ( \phi\circ\psi)}\arrow{d} && T\times U \arrow{d}\\
W \arrow{r}{\psi} & \tilde{U} \arrow{r}{\phi} & U,
\end{tikzcd}
$$
where vertical maps represent torus actions on $W$ and $U$, 
respectively. In particular, the action of $(S^1)^n$ on $W$ is the standard one. 
Such an orbifold together with a preferred orientation on each 
$T$-invariant suborbifold of codimension 1 is called a \emph{locally standard torus orbifold}. 
It is one of the immediate consequences of locally standard actions that the quotient $U/T$ is diffeomorphic to an open subset of 
\begin{equation}\label{eq_local_nbd_in_Q}
W/(S^1)^n\,\cong\, \RR^n_{\geq}\,\letbe\,\{ (x_1,\dots,x_n)\in \RR^n\mid x_i\geq 0\}.
\end{equation}

Weighted projective spaces are typical examples of locally standard torus orbifolds. 
Here we choose a particular half-dimensional torus actions. To be more precise, 
we consider a weighted projective space $\mbb{P}{(a_0, \dots, a_n)}$ 
as a quotient of $S^{2n+1}\subset \CC^{n+1}$ by a weighted $S^1$-action 
given by 
$$t\cdot (z_0, \dots, z_n)\,=\,(t^{a_0}z_0, \dots, t^{a_n}z_n)$$
with respect to some weight vector $(a_0, \dots, a_n)\in \NN^{n+1}$. 
Then, the standard $T^{n+1}$-action on $S^{2n+1}$ defines 
the residual action of an $n$-dimensional torus  $T^{n+1}/S^1\cong T^n$. 

\begin{example}\label{ex_tangential_repn_of_wCP}
Consider the weighted projective space $\mbb{P}{(1,2,3,6)}$, 
as a quotient space $S^7/S^1$, where $S^1$ acts on 
$S^7$ by 
$$t\cdot(z_1, z_2, z_3, z_4)\,=\,(tz_1, t^2z_2, t^3z_3, t^6z_4).$$
The standard $T^4$-action on $S^7$ induces a residual 
$T^4/S^1$-action on $\mbb{P}{(1,2,3,6)}$. 
Consider the following short exact sequence 
\begin{equation*}
\begin{tikzcd}
1\arrow{r} &S^1 \arrow{r}{\varpi} & T^4 \arrow{r}{\lambda} &T^3 \arrow{r}& 1,
\end{tikzcd}
\end{equation*}
where $\varpi(t)=(t,t^2, t^3, t^6)$ and 
$\lambda(t_1, t_2, t_3, t_4)=(t_1^{-2}t_2, t_1^{-3}t_3, t_1^{-6}t_4)$, so that 
$\ker(\lambda)={\rm im}(\varpi)$. 
Identifying ${\rm coker}(\varpi)$ with 
$T^3$ by taking a right splitting $\varrho \colon T^3 \to T^4$ defined by 
$\varrho(r_1, r_2, r_3)=(1, r_1, r_2, r_3)$, we calculate the orbifold tangential 
representations around each fixed point by a similar manner to
Example \ref{ex_wCP_111222}. The corresponding orbifold GKM-graph 
coincides with the one described in Figure \ref{fig_orb_tor_graph_CP1236}, 
which is indeed an orbifold torus graph. 
\end{example}

We continue this section by introducing a combinatorial model for 
a torus orbifold, and end up with studying its equivariant cohomology 
ring with integer coefficients. 

\begin{remark}
We refer to \cite{GKRW} for an exposition of torus orbifolds. 
The authors also discuss GKM-graphs for torus orbifolds, 
but they adopt $H^2(BT)$ as the target space of the axial function. 
The resulting equivariant cohomologies of torus orbifolds that they calculate are then taken with rational coefficients. 
\end{remark}

\subsection{Quotient construction}\label{subsec_canonical_model}
Given a $2n$-dimensional locally standard torus orbifold 
$\mathcal{X}=(X, \mathcal{U})$ we consider the orbit space $Q\letbe X/T$. 
The local standardness implies that any point in $Q$ has a neighborhood diffeomorphic to an open subset of \eqref{eq_local_nbd_in_Q}. The orbifold atlas $\mathcal{U}$ leads each of these neighborhoods to  fit together, so that the orbit space $Q$ 
has the structure of  a \emph{manifold with faces}, \cite[Definition 7.1.2]{BP-book}. 
The points in $Q$ 
corresponding to $0$-dimensional orbits and points corresponding 
to $(n-1)$-dimensional orbits are called \emph{vertices} and \emph{facets}, 
respectively. 

Let $\mathcal{F}(Q)\letbe \{F_1, \dots, F_m\}$ be the set of facets of 
$Q$ and $\pi\colon X\to Q$ be the orbit map. 
We choose elements 
\begin{equation}\label{eq_lambda_i}
\{\lambda_1, \dots, \lambda_m\}\subset H_2(BT)\,\cong\, \Hom (S^1, T)
\end{equation} 
such that $\lambda_i (S^1) \subset T$ is the isotropy subgroup of 
$\pi^{-1}(F_i)$, for $i=1, \dots, m$. We call each  $\pi^{-1}(F_i)$ a
\emph{characteristic suborbifold}.  
The local standardness of the torus action 
leads us to the next proposition. 
\begin{proposition}\label{prop_char_vector_lin_indep}
Let $v$ be a vertex of $Q$ and $F_{i_1}, \dots, F_{i_n}$ facets 
containing $v$. Then the set 
$\{\lambda_{i_1}, \dots, \lambda_{i_n}\}\subset H_2(BT)$ 
is linearly independent. 
\end{proposition}

Conversely, beginning with an $n$-dimensional manifold $Q$ with faces, we consider a function 
$$\lambda \colon \mathcal{F}(Q) \longrightarrow H_2(BT),$$  
satisfying the condition that  the vectors $\{\lambda(F_{i_1}), \dots, \lambda(F_{i_k})\}$ are
linearly independent whenever $F_{i_1}\cap \dots \cap F_{i_k} \neq \emptyset$. 
We call such a function $\lambda$ and a pair $(Q, \lambda)$ 
a \emph{characteristic function}, and a \emph{characteristic pair}, respectively. 
Due to technical reasons, we always assume that $Q$ is a $CW$-complex. 

Given a characteristic pair $(Q, \lambda)$, 
one can construct a torus orbifold $X(Q, \lambda)$ as follows:
\begin{equation}\label{eq_Definition_of_torus_orb}
X(Q, \lambda)\,\letbe\,(Q \times T) / \sim,
\end{equation}
where the equivalence relation $\sim$ is given by 
$(x,t) \sim (y,s)$  if and only if  $x=y$ and $t^{-1}s \in T_{F(x)}$.
Here, $F(x)$ denotes the face of $Q$ containing $x$ in its relative
interior and 
$T_{F(x)}$ is the torus generated by 
$\lambda(F_{i_1}), \dots, \lambda(F_{i_k})$, if 
$F(x)=F_{i_1} \cap \dots \cap F_{i_k}.$
We note that the $CW$-complex structure on $Q$ gives a 
$T$-CW complex structure on $X(Q, \lambda)$. 

The orbifold structure can be obtained in the same manner 
as described in \cite[Section 2.1]{PS} for quasitoric orbifolds.
We notice that $X(Q,\lambda)$ is equipped  with an $n$-dimensional 
torus $T$ action by multiplication on the second factor, and 
the orbit map is the projection of the first factor. 

Given a point $x\in X(Q, \lambda)$, 
let $F(x)$ be the face of $Q$
containing $\pi(x)$ in its relative interior. Then the isotropy subgroup 
of $x$ is the torus $T_{F(x)}$, which  yields the 
following proposition.
\begin{proposition}\label{prop_isotropy_conn}
Let $X(Q, \lambda)$ be a torus orbifold associated to a 
characteristic pair. Then every isotropy subgroup is 
connected.
\end{proposition}

The next theorem extends \cite[Lemma 4.5]{MP} and 
\cite[Lemma 2.2]{PS} to torus orbifolds, whose proof is similar to the proof of \cite[Proposition 1.8]{DJ} where we note that $H^{2}(Q;\mathbb{Z}^{n})$ is isomorphic to $[Q, BT]$, hence, the assumption for $H^{2}(Q;\mathbb{Z}^{n})$ being trivial implies 
that any principal $T$-bundle over $Q$ is trivial, which plays an important role 
in the proof. 
\begin{theorem}\label{thm_canonical_model}
Let $\mathcal{X}=(X, \mathcal{U})$ be a locally standard torus orbifold with 
$T$-action and $Q$ the orbit space such that $H^{2}(Q;\mathbb{Z}^n)$ is trivial.  
Let 
$\lambda\colon \mathcal{F}(Q) \to \ZZ^n \cong H_2(BT)$ 
be the characteristic function defined by $\lambda(F_i)=\lambda_i$
as in \eqref{eq_lambda_i}.  
Then there is an equivariant homeomorphism $f$ between $X$ and 
$X(Q, \lambda)$ such that the following diagram commutes:
\begin{equation*}
\begin{tikzcd}[column sep=small]
X(Q, \lambda) \arrow{rr}{f} \arrow{rd}[swap]{pr_1}&& X\arrow{ld}{\pi}\\
&Q&
\end{tikzcd}
\end{equation*}
where $pr_1$ is the projection onto the first factor and $\pi$  is the 
orbit map.
\end{theorem}

\begin{example}\label{ex_even_sphere}
Let $Q$ be the suspension of the $(n-1)$-simplex $\Delta^{n-1}$ for $n\geq 2$, 
see Figure \ref{fig_susp_of_simplex} for the case when $n=2$ and $3$. 
Then there are exactly $n$ facets, say $F_1, \dots, F_n$, on $Q$. 
Consider a characteristic function 
$$\lambda\colon \{F_1, \dots, F_n\} \longrightarrow H_2(BT)$$
defined by $\lambda(F_j)\letbe \sum_{i=1}^n a_{ij}\epsilon_i\in H_2(BT)$.
The integers $(a_{ij})_{1\leq i,j \leq n}$ form an $n\times n$ 
square matrix $\Lambda$ which we regard as an 
automorphism on $\mathbb{R}^{n}$. Now the space 
\begin{equation}\label{eq_orb_lens_sp}
(\Delta^{n-1} \times T)/\sim, 
\end{equation}
where the equivalence relation $\sim$ is same as in
\eqref{eq_Definition_of_torus_orb}, can be identified with a quotient 
of an odd sphere $S^{2n-1}$ by the action of the finite group 
$$\ker(\exp \Lambda\colon T \lra T).$$
The resulting space is 
known as an \emph{orbifold lens space}, see \cite[Section 3.3]{BSS}. 
Hence, the resulting torus orbifold $X(Q, \lambda)$ is exactly 
the suspension of an orbifold lens space. We refer to \cite{DKS-proc}
for the discussion on these spaces. 
In particular, if the determinant of $\Lambda$ is $\pm1$, i.e. 
the set $\{\lambda(F_1), \dots, \lambda(F_n)\}$ forms a $\ZZ$-basis 
of $H_2(BT)$, then the resulting space \eqref{eq_orb_lens_sp} is 
homeomorphic to  $S^{2n-1}$. Hence, the resulting 
torus orbifold $X(Q, \lambda)$ is homeomorphic to $S^{2n}$. 
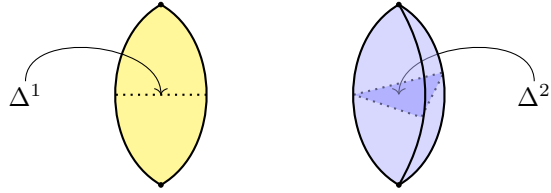
\begin{figure}
\begin{tikzpicture}[scale=0.6]
\draw[thick, fill=yellow!50] (0,0) [out=30, in=-30] to (0,4) [out=210, in=150] to (0,0);
\draw [fill] (0,0) circle [radius=0.05];
\draw [fill] (0,4) circle [radius=0.05];
\draw[thick, dotted] (-1,2)--(1,2);
\node at (-3,2) {$\Delta^1$};
\draw[->] (-3,2.3) [out=90, in=90] to (0,2);

\begin{scope}[xshift=150]
\draw[thick, dotted, fill=blue!30] (-1,2)--(1,2.5)--(0.5, 1.5)--cycle;

\draw[->] (3,2.3) [out=90, in=90] to (0,2);
\node at (3,2) {$\Delta^2$};

\draw[fill=blue!30, opacity=0.5] (0,0) [out=30, in=-30] to (0,4) [out=210, in=150] to (0,0);
\draw[thick] (0,0) [out=30, in=-30] to (0,4) [out=210, in=150] to (0,0);
\draw[thick] (0,0) [out=60, in=-60] to (0,4);
\draw [fill] (0,0) circle [radius=0.05];
\draw [fill] (0,4) circle [radius=0.05];

\end{scope}
\end{tikzpicture}
\caption{Suspensions of $\Delta^1$ and  $\Delta^2$.}
\label{fig_susp_of_simplex}
\end{figure}
\end{example}

\subsection{Equivariant cohomology ring of $X(Q, \lambda)$}
Given a characteristic pair $(Q, \lambda)$ for a torus orbifold, 
one can derive an orbifold GKM-graph as follows. Let 
\[\Gamma\,=\,(\mc{V}(\Gamma), \mc{E}(\Gamma))\] 
be the $1$-skeleton of  $Q$, 
which is an $n$-valent graph. 
We define a function 
\begin{equation*}\label{eq_axial_function}
\alpha \colon \mathcal{E}(\Gamma) \longrightarrow H^2(BT;\QQ)
\end{equation*}
by: if $e=F_{k_1} \cap \dots \cap F_{k_{n-1}}\in \mathcal{E}(\Gamma)$
with the initial vertex $i(e)=e\cap F_{k_n}$, then $\alpha(e)$ is defined by 
the following system of equations
\begin{equation}\label{eq_how_to_determine_axial_ftn}
\begin{cases} 
\left< \alpha(e), \lambda(F_{k_1}) \right>\,=\, \cdots \,=\, 
\left< \alpha(e), \lambda(F_{k_{n-1}}) \right>\,=\,0;\\
\left< \alpha(e), \lambda(F_{k_n}) \right>\,=\,1,  
\end{cases}
\end{equation}
where $\langle~,~ \rangle$ denotes the natural paring between
cohomology and homology. In particular, the integers $r_e$ (see Definition \ref{def_orb_GKM_graph}) are described below in the proof of Lemma \ref{lem_char_ftn_induces_axial_ftn}.

Intrinsically, each characteristic vector represents the $S^1$-subgroup of 
$T$ which acts trivially on 
characteristic suborbifold. Hence, each edge $e\in \mathcal{E}(\Gamma)$ 
represents an intersection of $n-1$ many characteristic suborbifolds. 
Indeed, this intersection is homeomorphic to a $2$-dimensional sphere and 
isomorphic to a spindle as an orbifold, 
fixed by a rank $n-1$ subgroup of $T$. The collection of equations in the
first line of \eqref{eq_how_to_determine_axial_ftn} contains the information 
of circle subgroups of $T$ which acts trivially on an invariant $2$-sphere, 
and the second equation tells us the the residual  $S^1$-action. 

\begin{lemma}\label{lem_char_ftn_induces_axial_ftn}
Let $(Q, \lambda)$ be a characteristic pair associated to a 
torus orbifold. Then the function $\alpha$ defined by
\eqref{eq_how_to_determine_axial_ftn}
is an axial function on $\Gamma$. In particular, the set 
$\{\alpha(e) \mid e\in \mathcal{E}_p(\Gamma), 
p\in \mathcal{V}(\Gamma)\}$ is linearly independent.
\end{lemma}
\begin{proof}
Let $\lambda(F_j)=\sum_{i=1}^n a_{ij}\epsilon_i\in H_2(BT)$, for 
$j=1, \dots, m$. Identifying $H_2(BT)$ and $\ZZ^n$ by associating $\ep_i$ with the $i$-th standard unit vector in $\ZZ^n$, we may write $\lambda(F_j)=(a_{1j}, \dots, a_{nj})\in \ZZ^n$. 
For each vertex $v$ of $Q$ with $v=F_{i_1} \cap \dots \cap F_{i_n}$, we write 
\begin{equation}\label{eq_lambda_v}
\Lambda_v \,\letbe\, \begin{bmatrix}\lambda(F_{i_1} )^t & \cdots &  \lambda(F_{i_n} )^t\end{bmatrix}. 
\end{equation} 

Given an oriented edge 
$e=F_{k_1} \cap \dots \cap F_{k_{n-1}}$ of $Q$, 
let $i(e)=e\cap F_{k_n}$ be the initial vertex and 
$t(e)= e\cap F_{k_n'}$ be the terminal vertex. Consider  
$\Lambda_{i(e)}$  and $\Lambda_{t(e)}$ with 
\[
r_e\letbe \big| \det \Lambda_{i(e)} \big| ~\text{ and } ~
r_{\bar{e}}\letbe \big| \det \Lambda_{t(e)} \big|,
\] 
respectively. 
Let $\alpha(e)_{\{\epsilon_1^\ast, \dots, \epsilon_n^\ast\}}$ be the coordinate expression of $\alpha(e)$ with respect to the basis $\{\epsilon_1^\ast, \dots, \epsilon_n^\ast\}\subset H^2(BT)$. Then the system of equations \eqref{eq_how_to_determine_axial_ftn} implies that 
\begin{align*}
\alpha(e)_{\{\epsilon_1^\ast, \dots, \epsilon_n^\ast\}} \cdot \Lambda_{i(e)}\,=\,(0,\dots,0,1) \quad \text{and} \quad  \alpha(\bar{e})_{\{\epsilon_1^\ast, \dots, \epsilon_n^\ast\}} \cdot \Lambda_{t(e)}\,=\,(0,\dots,0,1).
\end{align*}
Hence, we have 
\begin{align*}
\alpha(e)_{\{\epsilon_1^\ast, \dots, \epsilon_n^\ast\}}\,=\,(0,\dots,0,1) \cdot \Lambda_{i(e)}^{-1} \quad \text{and} \quad  \alpha(\bar{e})_{\{\epsilon_1^\ast, \dots, \epsilon_n^\ast\}}\,=\,(0,\dots,0,1)\cdot \Lambda_{t(e)}^{-1}, 
\end{align*}
which are the last rows of $\Lambda_{i(e)}^{-1}$ and $\Lambda_{t(e)}^{-1}$, respectively. Moreover, since the first $n-1$ columns of $\Lambda_{i(e)}$ and $\Lambda_{t(e)}$
are identical, 
the last row of $\text{adj}(\Lambda_{i(e)})$ agrees with the last row of $\text{adj}(\Lambda_{t(e)})$. This implies that 
\[
\det \Lambda_{i(e)}\cdot \alpha(e)_{\{\epsilon_1^\ast, \dots, \epsilon_n^\ast\}} \,=\,
\det \Lambda_{t(e)}\cdot \alpha(\bar e)_{\{\epsilon_1^\ast, \dots, \epsilon_n^\ast\}} .
\]
Hence, we get $r_e\alpha(e)=\pm r_{\overline{e}}\alpha(\overline{e})$ as desired. 

The second assertion follows immediately from the linear independence 
of characteristic vectors around a vertex. 
\end{proof}

When $(\Gamma, \alpha)$ is obtained from a characteristic pair $(Q, \lambda)$,
certain elements of weighted face ring $\ZZ[\Gamma, \alpha]$ 
can be immediately read off from $(Q, \lambda)$. 
\begin{proposition}\label{prop_elements_of_wFace_ring}
\hfill
\begin{enumerate}
\item For each facet $F$ of $Q$, let 
$\ell_F\letbe  {\rm lcm}\left\{ |\det \Lambda_v| ~\big|~ v\in \mathcal{V}(F) \right\}$, 
where $\Lambda_v$ is the $n\times n$ 
square matrix as defined in \eqref{eq_lambda_v}.
Then, $\ell_Fx_{F}$ is an element of $\ms{Z}_{\Gamma, \alpha}$. 
\item Let $\lambda(F_j)=\sum_{i=1}^n a_{ij}\epsilon_i\in H_2(BT)$ for 
facets $F_1, \dots, F_m$ of $Q$. Then 
\begin{equation}\label{eq_global_elt}
\bigg\{ \sum_{j=1}^{m} a_{ij}x_{F_j} ~\Big|~ i=1, \dots, n\bigg\}
\end{equation}
are elements of $\ms{Z}_{\Gamma, \alpha}$. 
%

\end{enumerate}
\end{proposition}

\begin{proof}
The proof of~(1)  is straightforward from the definition of rational Thom classes, see \eqref{eq_rational_thom_class}. In order to prove (2), it is enough to show that the restriction of \eqref{eq_global_elt} to each vertex $v\in \mathcal{V}(\Gamma)$ is an element of $H^2(BT)$. If $v=F_{k_1} \cap \dots \cap  F_{k_n}$ for some facets $F_{k_1}, \dots, F_{k_n}$ of $Q$, the definitions of $\mu$ and rational Thom classes (see \eqref{eq_mu:poly->H*BT} and \eqref{eq_rational_thom_class}, respectively) yield 
$$
\mu\Big(\sum_{j=1}^{m} a_{ij}x_{F_j} \Big) (v)\,=\, 
\sum_{\ell=1}^n a_{i k_\ell}\tau_{F_{k_\ell}} (v)\,=\,
\sum_{\ell=1}^n a_{i k_\ell}\alpha(e_{k_\ell}),
$$
where $e_{k_\ell}= 
F_{k_1} \cap \dots \cap F_{k_{\ell-1}} \cap F_{k_{\ell+1}}\cap  \dots \cap F_{k_n}$ 
with initial vertex $v$.

The system of equations 
\eqref{eq_how_to_determine_axial_ftn} together with a coordinate expression $\alpha(e_{k_\ell})_{\{\epsilon_1^\ast ,\dots, \epsilon_n^\ast\}}$ as in the proof of Lemma \ref{lem_char_ftn_induces_axial_ftn} 
allows us to continue the computation as follows:
\begin{align}\label{eq_global_element_proof}
\begin{split}
\sum_{\ell=1}^n a_{i k_\ell}\alpha(e_{k_\ell})&\,=\,
\sum_{\ell=1}^n a_{i k_\ell} \cdot 
[ 0~ \cdots ~   0 ~ \underset{\ell \text{-th}}{1} ~ 0~ \cdots ~0 ] \cdot 
\Lambda_v^{-1}\\
&\,=\,[a_{ik_1} ~ \cdots ~ a_{ik_n}] \cdot \Lambda_v^{-1}\\
&\,=\,[ 0~ \cdots ~   0 ~ \underset{i \text{-th}}{1} ~ 0~ \cdots ~0 ]\in 
\ZZ^n.
\end{split}
\end{align}
The final equality follows because $[a_{ik_1} ~ \cdots ~ a_{ik_n}]$ is
the $i$-th row of $\Lambda_v$. 

Recall that we identify $H^2(BT)$ with $\ZZ^n$ via the basis 
$\{\epsilon_1^\ast, \dots, \ep_n^\ast\}$. Hence, the vector 
$[0~\cdots~0~1~0~\cdots~0]$ in the last line of \eqref{eq_global_element_proof} 
corresponds to $\ep_i^\ast \in H^2(BT)$, which establishes the claim. 
\end{proof}

\begin{remark}\label{rmk_global_minimal_elements}
\hfill
\begin{enumerate}
\item We note that the computation \eqref{eq_global_element_proof} 
shows that, for each $i=1, \dots, n$, the image of elements in the set 
\eqref{eq_global_elt} via the map $\mu$ are the constant functions 
$\epsilon_i^\ast\colon \mc{V}(\Gamma) \to H^2(BT)$ defined by 
$v \mapsto \epsilon_i^\ast$, for every $v\in \mc{V}(\Gamma)$. 
\item If $\sum_{F\in \mc{F}^{(k)}}c_Fx_F$ is an element of $\ms{Z}_{\Gamma, \alpha}$, 
so is a constant multiple 
$$r\cdot \sum_{F\in \mc{F}^{(k)}}c_Fx_F\,=\,\sum_{F\in \mc{F}^{(k)}}rc_Fx_F$$
for some $r\in \ZZ\setminus \{0\}$. Hence, we may choose the 
vector $(\tilde{c}_F)_{F\in \mc{F}^{(k)}}$ (up to sign) having minimal 
length in $((c_F)_{F\in \mc{F}^{(k)}} \otimes_\ZZ \RR)\cap \ZZ^{|\mc{F}^{(k)}|}$, 
such that 
$$\sum_{F\in \mc{F}^{(k)}}\tilde{c}_Fx_F \in \ms{Z}_{\Gamma, \alpha}.$$
We call such elements 
\emph{minimal Thom classes}.
\item In~\cite{BFR} the equivariant cohomology of weighted projective spaces with integral coefficients was calculated and we observe that their \emph{Courant functions} correspond to the minimal Thom classes associated to facets of the underlying simplex. Since all \emph{global polynomials} appear in $\ZZ[\Gamma,\alpha]$, by Remark~\ref{rmk_global_minimal_elements}(1), all of the generators of their algebra have been covered.
\item When $X(Q, \lambda)$ is a projective toric variety \cite{CLS, Fulton}, namely $Q$ is a lattice polytope and $\lambda$ is defined to be primitive outward normal vectors of facets, then the set of linear elements in equation~\eqref{eq_subcollection_integrality_condition} coincides with the set ${\rm CDiv}_T(X(Q, \lambda))$ of $T$-invariant Cartier divisors of $X(Q, \lambda)$. We refer to \cite[Section 4.2]{CLS} and \cite[Section 3.3]{Fulton} for more details.
\end{enumerate}\end{remark}

\begin{example}
Consider a characteristic pair $(\Delta^3, \lambda)$ where 
characteristic function $\lambda$ is defined by 
$\lambda(F_1)=(-2, -3, -6),$ $\lambda(F_2)=(1,0,0),$ $\lambda(F_3)=(0,1,0)$ 
and $\lambda(F_4)=(0,0,1)$ for facets $F_i$, where $1 \leq i \leq 4$. 
The induced orbifold  torus graph $(\Gamma, \alpha)$ and  
the rational Thom classes of degree $2$ are described in
Figure \ref{fig_orb_tor_graph_CP1236} and  
Figure \ref{fig_rat_Thom_classes}, respectively. 
Proposition \ref{prop_elements_of_wFace_ring}-(1) says that 
$6x_{F_i}$, $i=1, \dots, 4$, are elements of $\ms{Z}_{\Gamma, \alpha}$. 
However, the minimal Thom classes corresponding to facets are 
$6x_{F_1}$, $3x_{F_2}$, $2x_{F_3}$ and $x_{F_4}$. One can also see that 
$-2x_{F_1}+x_{F_2},~-3x_{F_1}+x_{F_3}, -6x_{F_1}+x_{F_4}$ are 
elements of $\ms{Z}_{\Gamma, \alpha}$, whose images via $\mu$ are 
constant functions represented by $\ep_1^\ast,~\ep_2^\ast$ and $\ep_3^\ast$, 
respectively.
\end{example}

The main theorem of this section is as follows:
\begin{theorem}\label{thm_main_thm_of_paper}
Let $(Q, \lambda)$ be a characteristic pair and $X$ the associated 
torus orbifold with $H^{odd}(X)=0$. Then there is an isomorphism 
$$H^{*}_T(X) \,\cong\, \ZZ[\Gamma, \alpha]$$
of $H^\ast(BT)$-algebras, where $(\Gamma, \alpha)$ is the orbifold torus graph 
obtained from $(Q, \lambda)$. Furthermore, if $H^\ast(X)$ is free over $\ZZ$, then 
$\ZZ[\Gamma, \alpha]$ is finitely generated. 
\end{theorem}
\begin{proof}
A CW-complex structure on $Q$ defines a $T$-CW complex structure 
on $X$. Hence, the assumption $H^{odd}(X)=0$, together with 
Proposition \ref{prop_isotropy_conn}, leads us to apply Theorem \ref{thm_geometric_side}. 
 Finally, the first result follows from Theorem \ref{main-thm1}.

The proof of the second assertion is similar to the proof of 
\cite[Lemma 2.1]{BFR}. 
If $H^\ast(X)$ is free over $\ZZ$ and vanishes in odd degrees, 
the Serre spectral sequence of the fibration 
$$
\begin{tikzcd}
X\arrow{r} & ET\times_T X \arrow{r}{\pi} &  BT
\end{tikzcd}$$
degenerates at the $E_2$-page, 
and  there is a ring isomorphism 
$H_T^\ast(X)\cong H^\ast(X)\otimes H^\ast(BT)$ by the
Leray--Hirsh theorem. Hence, the ring 
generators come from either $H^\ast(X)$ or $H^2(BT)$. 
Recall from Remark \ref{rmk_global_minimal_elements}-(1) that 
the generators of $H^2(BT)$ are already elements of $\ms{Z}_{\Gamma, \alpha}$
which are degree $2$ elements of $\ZZ[\Gamma, \alpha]$. 
Moreover, the result of \cite[Theorem 1.1]{FP} implies that the set of elements in $H^\ast_T(X) (\cong \ZZ[\Gamma, \alpha])$ with degree less than or equal 
to the dimension of $X$ surjects onto 
$$H^\ast(X)\,\cong\, H^\ast_T(X)/ {\rm im}(\pi^\ast : H^{> 0}(BT)\longrightarrow H^\ast_T(X)),$$ 
which establishes 
the assertion. 
\end{proof}

From the algebraic point of view, it is not obvious that $\ZZ[\Gamma, \alpha]$
is finitely generated. The proof of the second assertion uses an application of geometry 
to answer this algebraic question. 


\begin{corollary}\label{cor_sing_cohom}
Let $X$ and $(\Gamma, \alpha)$ be as above. 
If $H^\ast(X)$ is free over $\ZZ$ and vanishes in all odd degrees, then 
$H^\ast(X)$ is isomorphic to $\ZZ[\Gamma, \alpha]/\mc{J}$, where 
$\mc{J}$ is the ideal generated by elements of \eqref{eq_global_elt}. 
\end{corollary}

\begin{remark}
Note that not every orbifold torus graph, as defined in Section~\ref{sec_torus_orb_graph}, originates 
from torus orbifolds as discussed in this section. For instance, the quotient of the rational graph equivariant cohomology $H^\ast_T(\Gamma, \alpha;\QQ)$ of the orbifold torus graph described 
in Figure \ref{fig_torus_orb_graph_without_geom} by $H^{>0}(BT;\QQ)$ 
fails Poincar\'{e} duality. Hence, results in \cite{Sat} imply that $(\Gamma, \alpha)$ cannot be induced from an orbifold whose odd degree cohomology vanishes.
\end{remark}

\begin{figure}[h]
\begin{tikzpicture}[scale=0.8]
\draw [fill] (0,0) circle [radius=0.05];
\draw [fill] (4,0) circle [radius=0.05];
\draw [fill] (4,4) circle [radius=0.05];
\draw [fill] (0,4) circle [radius=0.05];
\draw (0,0)--(4,0)--(4,4)--(0,4)--cycle;
\draw (0,0)--(4,4);
\draw (4,0)--(0,4);

\draw[thick, ->] (0,0)--(1,0);
\draw[thick, ->] (0,0)--(.7,.7);
\draw[thick, ->] (0,0)--(0,1);

\node[left] at (0,1) {\footnotesize$\frac{1}{3}\ep_2^\ast$};
\node[right] at (0.7, 0.7) {\footnotesize$\frac{1}{4}\ep_3^\ast$};
\node[below] at (1,0) {\footnotesize$\frac{1}{5}\ep_4^\ast$};
\node[left] at (-0.8,0.5) {\footnotesize$\frac{1}{2}\ep_1^\ast$};

\draw[thick, ->] (4,0)--(3,0);
\draw[thick, ->] (4,0)--(4-.7,.7);
\draw[thick, ->] (4,0)--(4,1);

\node[below] at (3,0) {\footnotesize$\frac{1}{5}\ep_4^\ast$};
\node[left] at (4-0.7, 0.7) {\footnotesize$\frac{1}{4}\ep_3^\ast$};
\node[right] at (4,1) {\footnotesize$\frac{1}{3}\ep_2^\ast$};
\node[right] at (4.8,0.5) {\footnotesize$\frac{1}{2}\ep_1^\ast$};

\draw[thick, ->] (0,4)--(0,3);
\draw[thick, ->] (0,4)--(.7,4-.7);
\draw[thick, ->] (0,4)--(1,4);

\node[left] at (0,3) {\footnotesize$\frac{1}{3}\ep_2^\ast$};
\node[right] at (0.7, 4-0.7) {\footnotesize$\frac{1}{4}\ep_3^\ast$};
\node[above] at (1,4) {\footnotesize$\frac{1}{5}\ep_4^\ast$};
\node[left] at (-0.8,3.5) {\footnotesize$\frac{1}{2}\ep_1^\ast$};

\draw[thick, ->] (4,4)--(3,4);
\draw[thick, ->] (4,4)--(4-.7,4-.7);
\draw[thick, ->] (4,4)--(4,3);

\node[above] at (3,4) {\footnotesize$\frac{1}{5}\ep_4^\ast$};
\node[left] at (4-0.7, 4-0.7) {\footnotesize$\frac{1}{4}\ep_3^\ast$};
\node[right] at (4,3) {\footnotesize$\frac{1}{3}\ep_2^\ast$};
\node[right] at (4.8,3.5) {\footnotesize$\frac{1}{2}\ep_1^\ast$};

\begin{scope}[xscale=0.6]
\draw (0,0) arc (270:90:2);
\draw[thick, -> ] (0,0) arc (270:225:2);
\draw[thick, -> ] (0,4) arc (90:135:2);
\end{scope}

\begin{scope}[xscale=0.6]
\draw (20/3,4) arc (90:0:2) arc (0:-90:2);
\draw[thick, -> ] (20/3,4) arc (90:45:2);
\draw[thick, -> ] (20/3,0) arc (270:315:2);
\end{scope}

\end{tikzpicture}
\caption{An orbifold torus graph without geometric origin.}
\label{fig_torus_orb_graph_without_geom}
\end{figure}
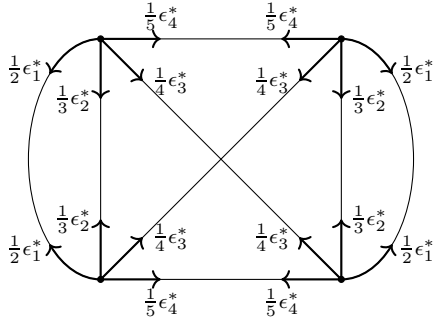

\section{An application: $4$-dimensional torus orbifolds}
In this section we consider the case when $Q$ is an $m(\geq 2)$-gon, namely a $2$-dimensional manifold with faces having $m$ vertices and $m$ sides. 
We begin by setting up the following notation. See Figure \ref{fig_char_ftn_for_polygon}, where the left picture gives the full description for the case when $m=2$.
\begin{itemize}
\item $\mathcal{F}^{(1)}=\{F_1, \dots, F_m\}$, the set of facets which are edges in this case. 
\item $\mathcal{F}^{(0)}=\{v_1, \dots, v_m\}$, the set of vertices, where 
$v_k= F_k \cap F_{k+1}.$
\item $\lambda \colon \mathcal{F}^{(1)} \to H_2(BT)$, 
a characteristic function and we write  
$$\lambda(F_k)\,\letbe\, a_k \epsilon_1 +  b_k \epsilon_2, \quad \text{ for some } a_k, b_k \in \ZZ. $$
\item $D_k\letbe \det \Lambda_{v_{k}}= \det[\lambda(F_k)^{t}\ \lambda(F_{k+1})^{t}]= a_k b_{k+1} - b_k a_{k+1}$.
\end{itemize}
Here, we identify the index $m+1$ with $1$. Hence, $v_m=F_m\cap F_{1}$ and $D_m=a_m b_{1} - b_m a_{1}$. 

\begin{figure}
\begin{tikzpicture}[scale=0.4]
\begin{scope}
\draw[fill=yellow] (0,0) to [out=60, in=180] (3,2) to [out=0, in=120] (6,0) to [out=240, in=0] (3,-2) to [out=180, in=300] (0,0);
\filldraw (0,0) circle (3pt);
\filldraw (6,0) circle (3pt);
\node[left] at (0,0) {\footnotesize$v_1$};
\node[right] at (6,0) {\footnotesize$v_2$};
\node[below] at (3, 2) {\footnotesize$F_1$};
\node[above] at (3,-2) {\footnotesize$F_2$};
\node[above] at (3,2) {\scriptsize$a_1\epsilon_1 +  b_1\epsilon_2$};
\node[below] at (3,-2) {\scriptsize$a_2\epsilon_1 +  b_2\epsilon_2$};
\end{scope}

\begin{scope}[xshift=350, yshift=-30, scale=1.5]
\draw[fill=yellow] (-1/2,-1)--(1,2)--(5,2)--(6+1/2,-1);
\draw[fill] (1,2) circle [radius=0.1];
\draw[fill] (5,2) circle [radius=0.1];
\node[below] at (1.3,2) {\scriptsize$v_{k}$};
\node[below] at (4.6,2) {\scriptsize$v_{k-1}$};
\node[right] at (0,0) {\scriptsize$F_{k+1}$};
\node[below] at (3, 2) {\scriptsize$F_k$};
\node[left] at  (6,0)  {\scriptsize$F_{k-1}$};
\node[rotate=64] at (-0.5,0.5) {\scriptsize$a_{k+1}\epsilon_1 +  b_{k+1}\epsilon_2$};
\node[rotate=-62] at (6.5,0.5) {\scriptsize$a_{k-1}\epsilon_1 +  b_{k-1}\epsilon_2$};
\node[above] at (3, 2) {\scriptsize$a_{k}\epsilon_1 +  b_{k}\epsilon_2 $};
\end{scope}

\end{tikzpicture}
\caption{Facets and vertices of $2$-gon and $m(\geq 3)$-gon.}
\label{fig_char_ftn_for_polygon}
\end{figure}
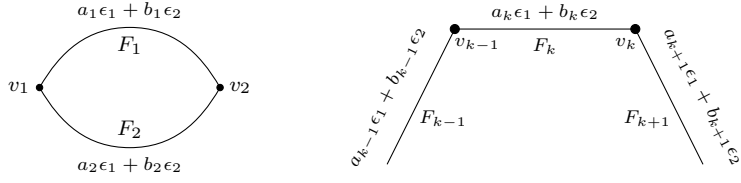

Following the result of \cite[Corollary 4.1]{KMZ}, 
$H^\ast(X)$ is concentrated in even degrees if and only if 
\begin{equation}\label{eq_iff_cond_for_Hodd=0_in_4_dim}
{\rm span}_\ZZ \{ a_k\epsilon_1 +  b_k\epsilon_2 \mid k=1, \dots , m\} \,=\, H_2(BT), 
\end{equation}
or equivalently, 
\begin{equation}\label{eq_iff_cond_for_Hodd=0_in_4_dim_BSS}
{\rm gcd} \{D_1, \dots, D_m\}\,=\,1,
\end{equation}
see \cite[Example 4.4]{BSS} and~\cite[Section~4.1]{BNSS} for a more general discussion.

Hence, we may apply Theorem \ref{main-thm1} to $4$-dimensional torus orbifolds
satisfying \eqref{eq_iff_cond_for_Hodd=0_in_4_dim} or \eqref{eq_iff_cond_for_Hodd=0_in_4_dim_BSS}. 
Preceding the computation of equivariant cohomology, we first introduce 
a certain $GL_2(\RR)$-representation, which will play an important role 
in our calculation.

%

\subsection{A $GL_2(\RR)$-representation}
In this subsection, we  identify $GL_2(\RR)$ with the general linear group $GL(H_2(BT^2;\RR))$ by taking $\{\epsilon_1, \epsilon_2\}$ as a basis of $H_2(BT^2;\RR)$. 
Let  $V_n\colonequals \bigoplus_{t=0}^n\RR r^{n-t}s^t$ be a vector space of homogeneous polynomials of degree $n$ with variables $r$ and $s$. Then, $V_n$ is isomorphic to $H^{2n}(BT^2;\RR)$ as a vector space by associating $r, s$ with $\epsilon_1^\ast, \epsilon_2^\ast \in H^2(BT^2;\RR)$, respectively. 
Observe that $V_n$ is a $GL_2(\RR)$-representation space with respect to the action $\Psi \colon GL_2(\RR) \times V_n \to V_n$ defined by 
$$
\Psi \left(\begin{bmatrix} a & b \\ c & d\end{bmatrix} , f(r,s)\right) \,=\, f(ar+bs, cr+ds),
$$
which gives a group homomorphism 
\begin{equation}\label{eq_sharp_operation}
\Phi^{(n)}  \colon GL_2(\RR) \lra GL(V_n).
\end{equation}
Note that $\Phi^{(1)}$ is the identity with respect to the basis $\{\epsilon_1, \epsilon_2\}$ of $H_2(BT^2;\RR)$ and its dual basis $\{\epsilon_1^\ast, \epsilon_2^\ast\}$ of $H^2(BT^2;\RR)\cong V_1$.

\begin{example}\label{ex_when_n=2}
When $n=2$, the ordered basis $\{r^2, rs, s^2\}$ of $V_2$ gives the following matrix presentation of \eqref{eq_sharp_operation} 
$$\Phi^{(2)}\left(\begin{bmatrix} a & b \\ c & d \end{bmatrix} \right)=
\begin{bmatrix} a^2&		2ab&		b^2\\
ac&ad+bc&bd\\
c^2&2cd&d^2\end{bmatrix}.$$
Indeed, 
\begin{align*}
\Psi \left(\begin{bmatrix} a & b \\ c & d\end{bmatrix} , r^2\right) &= (ar+bs)^2= a^2r^2 + 2ab rs + b^2s^2; \\
\Psi \left(\begin{bmatrix} a & b \\ c & d\end{bmatrix} , rs\right) &= (ar+bs)(cr+ds)= ac r^2 + (ad+bc) rs + bd s^2; \\
\Psi \left(\begin{bmatrix} a & b \\ c & d\end{bmatrix} , s^2\right) &=(cr+ds)^2 = c^2r^2+2cd rs + d^2 s^2.
\end{align*}
\end{example}

\subsection{Equivariant cohomology ring}

Each pair of adjacent facets $F_k, F_{k+1}$ defines a $(2\times 2)$-matrix 
\begin{align*}
\Lambda_{k} &\colonequals \left[\begin{array}{c|c} \lambda(F_k)  & \lambda(F_{k+1}) \end{array}\right]
= \left[\begin{array}{cc}a_k & a_{k+1} \\b_k & b_{k+1}\end{array}\right] \text{for } 1\leq k \leq m-1;\\
\Lambda_{m} &\colonequals  \left[\begin{array}{c|c} \lambda(F_m)  & \lambda(F_{1}) \end{array}\right]
= \left[\begin{array}{cc}a_m & a_{1} \\b_m & b_{1}\end{array}\right].
\end{align*}
Using the $(n+1)\times (n+1)$ matrix  $\Phi^{(n)}(\Lambda_k)$, we define a square matrix $\tilde\Lambda_k^{(n)}$ of size 
$(nm \times nm)$ as follows:
\begin{align*}
\tilde\Lambda_1^{(n)}& \colonequals  
\left[ \begin{array} {c|c}
&\\
\Phi^{(n)}(\Lambda_{1}) & \mathbf{0}\\ 
& \\ \hline 
& \\
\mathbf{0} & I_{n(m-1)-1}\\
& 
\end{array}\right]; \\
\tilde\Lambda_k^{(n)}& \colonequals 
\left[ \begin{array}{c|c|c}
&&\\
I_{n(k-1)} & \mathbf{0}&\mathbf{0} \\
&&\\ \hline
&&\\
\mathbf{0}&\Phi^{(n)}(\Lambda_{k})&\mathbf{0}\\
&&\\ \hline
&&\\
\mathbf{0}&\mathbf{0}&I_{n(m-k)-1}\\
&&\end{array}\right],\quad \text{for } 2\leq k \leq m-1;\\
\tilde\Lambda_m^{(n)}& \colonequals 
\left[\begin{array}{c|ccc|ccc}
 &&& &&& \\
\mathbf{0} &&I_{n(m-1)-1}&&&\mathbf{0}& \\
&&&&&&\\ \hline 
&&&&&&\\
\mathfrak{s}_{n+1}&&\mathbf{0}&&\mathfrak{s}_1&\cdots&\mathfrak{s}_n\\
&&&&&&
\end{array}\right],
\end{align*}
where $\mathbf{0}$ and $\mathfrak{s}_i$  denote the zero matrix of size fitting in the block and the $i$-th column of $\Phi^{(n)}(\Lambda_{m})$, respectively. We define $\mathbb{L}_k^{\deg 2n}$ to be the sublattice of $\ZZ^{nm}$ spanned by the row vectors of $\tilde{\Lambda}_k^{(n)}$.

%
%
\begin{theorem}\label{thm_4-dim_equiv_cohom}
Let $X\letbe X(Q, \lambda)$ be a 4-dimensional torus orbifold satisfying condition
\eqref{eq_iff_cond_for_Hodd=0_in_4_dim} or \eqref{eq_iff_cond_for_Hodd=0_in_4_dim_BSS}. 
Then $H^\ast_T(X)$ is generated by the union of 
\begin{enumerate}
\item[(i)] $\left\{  \sum_{1\leq i \leq m} c_{i}x_{F_i} 
~\Big| ~ (c_1, \dots, c_m)\in \bigcap_{k=1}^m \mathbb{L}_{k}^{\deg 2}\right\}$;
\item[(ii)] $\left\{  \sum_{\substack{t=0, 1 \\ 1\leq i \leq m }} c_{t, i}x_{F_i}^{2-t} x_{F_{i+1}}^t
~\Big| ~ (c_{0,1},c_{1,1}, c_{0,2}, c_{1,2},  \dots, c_{1,m})\in \bigcap_{k=1}^m\mathbb{L}_{k}^{\deg 4}\right\}$. \label{item_deg4}
\end{enumerate}
\end{theorem}
\begin{proof}
The orbifold torus graph $(\Gamma, \alpha)$ associated to 
$(Q, \lambda)$ is given by  
$$\mc{V}(\Gamma)\,=\,\mc{F}^{(0)}\,=\,\{v_1, \dots, v_m\} \quad\text{and}\quad
\mc{E}(\Gamma)\,=\,\{e_k, \bar{e}_k\mid k=1, \dots, m\},$$ 
where $e_k=F_k$ with initial vertex $v_{k-1}$.
The system of equations \eqref{eq_how_to_determine_axial_ftn} yields 
\begin{align*}
\alpha(e_k)\,=\,\frac{1}{D_{k-1}}\left(b_k \epsilon_1^\ast - a_k \epsilon_2^\ast\right)
\quad\text{and}\quad 
\alpha(\bar{e}_k)\,=\,\frac{1}{D_{k}}\left(-b_k \epsilon_1^\ast + a_k \epsilon_2^\ast\right),
\end{align*}
which gives the associated orbifold torus graph $(\Gamma, \alpha)$,
see Figure \ref{fig_polygon_char_pair_graph}. 
\begin{figure}
\begin{tikzpicture}[scale=0.8]
\draw (0,0)--(1,2)--(5,2)--(6,0);
\draw[fill] (1,2) circle [radius=0.1];
\draw[fill] (5,2) circle [radius=0.1];
\node[left] at (1,2) {\scriptsize$v_{k}$};
\node[right] at (5,2) {\scriptsize$v_{k-1}$};

\draw[very thick, ->] (1,2)--(2,2); \node[below] at (2,1.9) {\scriptsize${\bar{e}_k}$};
\draw[very thick, ->] (5,2)--(4,2); \node[below] at (4,1.9) {\scriptsize${e_k}$};
\draw[very thick, ->] (1,2)--(0.5,1); \node[right] at (0.7,1) {\scriptsize${e_{k+1}}$};
\draw[very thick, ->] (5,2)--(5.5,1); \node[right] at (5.6,1) {\scriptsize${\bar{e}_{k-1}}$};

\node[above] at (5,2.2) {\scriptsize$\frac{1}{D_{k-1}} \left( b_k \epsilon_1^\ast -a_k \epsilon_2^\ast \right)$};
\node[above] at (1,2.2) {\scriptsize$\frac{1}{D_{k}} \left( -b_k \epsilon_1^\ast +a_k \epsilon_2^\ast \right)$};
\node[below] at (6,0.9) {\scriptsize$\frac{1}{D_{k-1}} \left( -b_{k-1} \epsilon_1^\ast +a_{k-1} \epsilon_2^\ast \right)$};
\node[below] at (0,0.9) {\scriptsize$\frac{1}{D_{k}} \left( b_{k+1} \epsilon_1^\ast -a_{k+1} \epsilon_2^\ast \right)$};

\end{tikzpicture}
\caption{The orbifold torus  graph.}
\label{fig_polygon_char_pair_graph}
\end{figure}

In the associated weighted ring $\ZZ[\Gamma, \alpha]=\ms{Z}_{\Gamma, \alpha}/\mc{I}$, 
the ideal $\mc{I}$ allows us to express an arbitrary element of degree $2n$ 
in $\ZZ[\Gamma, \alpha]$ by 
\begin{equation}\label{eq_arb_elem_in_w_face_ring}
 \sum_{t=0}^{n-1}\sum_{i=1}^m c^{(n)}_{t, i}x_{F_i}^{n-t}x_{F_{i+1}}^t. 
\end{equation}
Recall from \eqref{eq_subcollection_integrality_condition} that 
an element \eqref{eq_arb_elem_in_w_face_ring} is indeed an element of $\ZZ[\Gamma, \alpha]$
if and only if 
\begin{equation}\label{eq_4_dim_gen_elem}
\left(
c_{0,k}^{(n)}\tau_{F_k}^n + c_{1, k}^{(n)}\tau_{F_k}^{n-1}\tau_{F_{k+1}} + \cdots 
+ c_{n-1, k}^{(n)}\tau_{F_k}\tau_{F_{k+1}}^{n-1} + c_{0, k+1}^{(n)}\tau_{F_{k+1}}^n\right) \Big|_{v_k} 
\end{equation}
is an element of $H^{2n}(BT)$ for each vertex $v_k$, $k=1, \dots , m$. 
Notice that the restrictions of the rational Thom classes $\tau_{F_k}$ and 
$\tau_{F_{k+1}}$ to $v_k$ are given by 
\begin{equation}\label{eq_rat_Thom_rest_to_v_k}
\tau_{F_k}\big|_{v_k}\,=\,\frac{1}{D_{k}} \left( b_{k+1} \epsilon_1^\ast -a_{k+1} \epsilon_2^\ast \right) \quad\text{and}\quad 
\tau_{F_{k+1}}\big|_{v_k}\,=\, \frac{1}{D_{k}} \left( -b_k \epsilon_1^\ast +a_k \epsilon_2^\ast \right), 
\end{equation}
respectively. Hence, by identifying
$$H^{2n}(BT)\,\cong\, \bigoplus_{t=0}^{n}\ZZ (\epsilon_1^\ast)^{n-t}(\epsilon_2^\ast)^t,$$
and plugging \eqref{eq_rat_Thom_rest_to_v_k} into 
\eqref{eq_4_dim_gen_elem}, we conclude that \eqref{eq_4_dim_gen_elem} is an element 
of $H^{2n}(BT)$ if and only if 
$$(c_{0, k}^{(n)}, \dots, c_{n-1, k}^{(n)}, c_{0, k+1}^{(n)}) \cdot 
\Phi^{(n)}\left( \Lambda_k^{-1} \right) \in \ZZ^{n+1},$$
which is equivalent to saying that 
$(c_{0, k}^{(n)}, \dots, c_{n-1, k}^{(n)}, c_{0, k+1}^{(n)})$ is an element in the sublattice of $\ZZ^{n+1}$ 
spanned by the row vectors of $\Phi^{(n)}(\Lambda_k)$. 
It is again equivalent to say that 
\begin{equation}\label{eq_all_entries_in_4-dim}
(c_{0,1}^{(n)}, \dots, c_{n-1,1}^{(n)}, \dots, c_{0,k}^{(n)}, \dots, c_{n-1, k}^{(n)}, \dots, c_{0,m}^{(n)}, \dots,  c_{n-1,m}^{(n)}) 
\end{equation}
is an element of $\mathbb{L}_{k}^{\deg 2n}$, because there is no condition for entries in \eqref{eq_all_entries_in_4-dim} except for $c_{0,k}^{(n)}, \dots, c_{n-1, k}^{(n)}$ and $c_{0, k+1}^{(n)}$. 
The same computation for all other vertices concludes that \eqref{eq_all_entries_in_4-dim} 
has to be an element of $\bigcap_{k=1}^m \mathbb{L}^{\deg 2n}_k$. 

Finally, the assumption \eqref{eq_iff_cond_for_Hodd=0_in_4_dim} or 
\eqref{eq_iff_cond_for_Hodd=0_in_4_dim_BSS} implies that $H^\ast(X)$ is free over $\ZZ$. 
Moreover, the second assertion of Theorem \ref{thm_main_thm_of_paper} concludes that 
$\ZZ[\Gamma, \alpha]$ is generated by elements of degree less than or equal to $4$. In degree $2$,  \eqref{eq_arb_elem_in_w_face_ring} becomes
$\sum_{i=1}^m c^{(1)}_{0, i}x_{F_i}$ and its coefficients satisfy $(c^{(1)}_{0, 1}, \dots, c^{(1)}_{0, m})\in \bigcap_{k=1}^m\mathbb{L}_k^{\deg2}$, which coincides with (i) by writing $c_i\colonequals c_{0,i}^{(1)}$ for $i=1, \dots, m$. In degree $4$, an element \eqref{eq_arb_elem_in_w_face_ring} can be written by $\sum_{\substack{t=0,1 \\ 1 \leq i \leq m}}c^{(2)}_{t,i}x_{F_i}^{2-t}x_{F_i+1}^t$ such that $(c^{(2)}_{0,1}, c^{(2)}_{1,1}, c^{(2)}_{0,2},c^{(2)}_{1,2},\dots, c^{(2)}_{0,m},c^{(2)}_{1,m}) \in \bigcap_{k=1}^m\mathbb{L}_k^{\deg4}$. This also agrees with (ii) by writing $c_{t,i}\colonequals c^{(2)}_{t,i}$ for $t=0,1$ and $i=1, \dots, m$. Hence, the proof is completed. 

\end{proof}

We finish this paper by considering a specific example: the $4$-dimensional torus orbifold\footnote{It is also a \emph{toric orbifold} in the sense of \cite[Section 7]{DJ}.} associated to a Cartan matrix of type $A$, see Figure \ref{fig_toric_orb_cartan_matrix} for the characteristic pair $(Q, \lambda)$. We refer the readers to \cite{Bru}. 

Notice that $(Q, \lambda)$ in Figure \ref{fig_toric_orb_cartan_matrix} satisfies \eqref{eq_iff_cond_for_Hodd=0_in_4_dim} and \eqref{eq_iff_cond_for_Hodd=0_in_4_dim_BSS}. Hence the cohomology of the  corresponding torus orbifold $X(Q, \lambda)$ is torsion free and concentrated in even degrees. In Example \ref{ex_toric_orb_cartan_mat}, we apply Theorem \ref{thm_main_thm_of_paper}, in particular Theorem \ref{thm_4-dim_equiv_cohom}, to calculate $H_T^\ast(X(Q, \lambda))$. We can also apply Corollary \ref{cor_sing_cohom} to obtain $H^\ast(X(Q,\lambda))$. In Example \ref{ex_toric_orb_cartan_mat_sing_cohom}, we calculate explicit generators of $H^\ast(X(Q,\lambda))$ and their relations. 

\begin{figure}
\begin{tikzpicture}
\draw[fill=yellow] (0,0)--(2,0)--(2,2)--(0,2)--cycle;
\draw[fill] (0,2) circle [radius=0.05]; 
\draw[fill] (2,2) circle [radius=0.05];
\draw[fill] (2,0) circle [radius=0.05];
\draw[fill] (0,0) circle [radius=0.05];

\node[left] at (0,0) {\scriptsize$v_1$};
\node[right] at (2,0) {\scriptsize$v_2$};
\node[right] at (2,2) {\scriptsize$v_3$};
\node[left] at (0,2) {\scriptsize$v_4$};

\node[right] at (0,1) {\scriptsize$F_1$};
\node[above] at (1,0) {\scriptsize$F_2$};
\node[left] at (2,1) {\scriptsize$F_3$};
\node[below] at (1,2) {\scriptsize$F_4$};

\node[left] at (0,1) {\scriptsize$-2\epsilon_1 + \epsilon_2$};
\node[below] at (1,0) {\scriptsize$\epsilon_1 -2 \epsilon_2$};
\node[right] at (2,1) {\scriptsize$\epsilon_1$};
\node[above] at (1,2) {\scriptsize$\epsilon_2$};

\end{tikzpicture}
\caption{A characteristic pair.}
\label{fig_toric_orb_cartan_matrix}
\end{figure}

\begin{example}[Equivariant cohomology]\label{ex_toric_orb_cartan_mat}
To apply Theorem \ref{thm_4-dim_equiv_cohom}, we begin with calculating matrices $\tilde\Lambda_k^{(n)}$  for $n=1,2$, and $1\leq k \leq 4$.  
When $n=1$, the homomorphism $\Phi^{(1)}\colon GL_2(\RR) \to GL_2(\RR)$ is the identity. Hence, we have 
\begin{align*}
\begin{array}{ll}
\tilde\Lambda_1^{(1)}=\left[ \begin{array}{rr|rr}
-2 &1 &0 &0 \\
1 &-2& 0& 0 \\ \hline
0 &0 &1 &0 \\
0 & 0& 0& 1
\end{array}\right] 
& 
\tilde\Lambda_2^{(1)}=\left[ \begin{array}{r|rr|r}
1 & 0 & 0 & 0 \\\hline
0 &1 &1 &0 \\
0 &-2& 0& 0 \\ \hline
0 &0&0&1
\end{array}\right], 
\\
& 
\\
\tilde\Lambda_3^{(1)}=\left[ \begin{array}{rr|rr}
1 & 0&0&0\\
0&1&0&0\\ \hline
 0&0&1 &0 \\
0&0& 0& 1 
\end{array}\right] ,
&   
\tilde\Lambda_4^{(1)}=\left[ \begin{array}{r|rr|r}
0&1&0&0\\
0&0&1&0\\ \hline
-2 &0 &0 &0 \\
1 &0& 0& 1 \end{array}\right]. 
\end{array}
\end{align*}
Now, we have sublattices $\mathbb{L}_k^{\deg 2}$ of $\ZZ^4$ generated by the row vectors of $\tilde\Lambda_k^{(1)}$ for each $1\leq k \leq 4$,
whose intersection $\bigcap_{k=1}^4 \mathbb{L}_k^{\deg 2}$ can be generated by 
\begin{equation}\label{eq_deg_2_generators}
\{(-2, 1, 1, 0), (1, -2, 0,1), (0,0,2,0), (0,0,0,2)\}.
\end{equation}
Hence, $H_T^2(X(Q, \lambda))$ is generated by 
\begin{align}
\begin{split}\label{eq_deg2_generator_zeta}
\zeta_1&\colonequals -2x_{1}+x_{2}+x_{3},\\
\zeta_2&\colonequals x_{1}-2x_{2}+x_{4},\\
\zeta_3&\colonequals 2x_{3},\\
\zeta_4&\colonequals 2x_{4},
\end{split}
\end{align}
where we denote $x_i:=x_{F_i}$ for $1\leq i \leq 4$ for simplicity. We notice that $\{ \zeta_1, \zeta_2\}$ are elements described in \eqref{eq_global_elt}. 
\begin{remark}
An explicit example of \eqref{eq_deg_2_generators} needs some tedious calculation, or one can use a computer program, for example the module of ``Toric lattices'' of SAGE, see \cite{SAGE-lattice} and \cite{sage}.
\end{remark}

When $n=2$, using the computation in Example \ref{ex_when_n=2}, we have the following four $(8\times 8)$-matrices
\begin{align*}
\begin{array}{ll}
\tilde\Lambda_1^{(2)}=\left[ \begin{array}{rrr|rrr}
4& -4& 1& & &  \\
-2& 5& -2& &\mathbf{0} & \\
1& -4& 4& & & \\ \hline
&&&&&\\
&\mathbf{0}&&&I_5&\\
&&&&&
\end{array}\right],
&
\tilde\Lambda_2^{(2)}=\left[ \begin{array}{rrr|rrr|rrr}
&&&&&&&&\\
&I_2&&&\mathbf{0}&&&\mathbf{0}&\\
&&&&&&&&\\  \hline
&&&  1&  2&  1&  &  &  \\
&\mathbf{0}&& -2& -2&  0&  & \mathbf{0} &  \\
&&&  4&  0&  0&  &  &   \\ \hline
&&&&&&&& \\
&\mathbf{0}&&&\mathbf{0}&&&I_3& \\
&&&&&&&& \\
\end{array}\right],
\\
\tilde\Lambda_3^{(2)}=\left[ \begin{array}{rrr|rrr|r}
&&&&&&\\
&I_4&&&\mathbf{0}&&\mathbf{0}\\
&&&&&&\\ \hline
&&&1&0&0&\\
&\mathbf{0}&&0&1&0&\mathbf{0}\\
&&&0&0&1& \\ \hline
&\mathbf{0}&&&\mathbf{0}&&1
\end{array}\right],
&
\tilde\Lambda_4^{(2)}=\left[ \begin{array}{c|c|c}
&&\\
\mathbf{0}&I_5&\mathbf{0} \\
&&\\ \hline
\begin{array}{r} 4\\-2\\1 \end{array} & \mathbf{0} & \begin{array}{rr} 0&0\\0&-2\\1&-2\end{array}
\end{array}\right].
\end{array}
\end{align*}
Now, we have sublattices $\mathbb{L}_k^{\deg 4}$ of $\ZZ^8$ generated by row vectors of $\tilde\Lambda_k^{(2)}$ above for $1\leq k\leq 4$, respectively. 
Finally, one can find generators of the intersection $\bigcap_{k=1}^4 \mathbb{L}_4^{\deg 4}$ as follows:
\begin{align*}
\{(1, 2, 1, 0, 3, 0, 1, 2), (0, 3, 3, 0, 1, 0, 0, 0), (0, 0, 9, 0, 3, 0, 0, 0), (0, 0, 0, 2, 2, 0, 0, 0), \\(0, 0, 0, 0, 4, 0, 0, 0), (0, 0, 0, 0, 0, 1, 0, 0), (0, 0, 0, 0, 0, 0, 2, 2), (0, 0, 0, 0, 0, 0, 0, 4)\},
\end{align*}
which means that $H_T^4(X(Q, \lambda))$ is generated by 
\begin{equation}\label{eq_deg4_generators_eta}
\begin{array}{l}
\eta_1 \colonequals x_{1}^2 + 2x_{1}x_{2} +x_{2}^2 + 3x_{3}^2 +x_{4}^2+2x_{1}x_{4},\\
\eta_2 \colonequals 3x_{1}x_{2}+ 3x_{2}^2 +x_{3}^2,\\
\eta_3 \colonequals 9x_{2}^2+3x_{3}^2,\\
\eta_4 \colonequals 2x_{2}x_{3} + 2x_{3}^2,\\
\eta_5 \colonequals 4x_{3}^2,\\
\eta_6 \colonequals x_{3}x_{4},\\
\eta_7 \colonequals 2x_{4}^2+ 2x_{1}x_{4},\\
\eta_8 \colonequals 4x_{1}x_{4}.
\end{array}
\end{equation}

%
\end{example}

\begin{remark}
One can also interpret generators $\zeta_1, \dots, \zeta_4$ in \eqref{eq_deg2_generator_zeta} and  $\eta_1, \dots \eta_8$ in \eqref{eq_deg4_generators_eta} as elements in $H^\ast_T(\Gamma, \alpha)$ via Theorem \ref{main-thm1}. Recall that the map $\nu$ in \eqref{eq_comm_diag_face_ring_graph_cohom} sends $x_{F_i}$ to $\tau_{F_i}$ for $i=1, \dots, 4$, which are illustrated in Figure \ref{fig_rational_thom_for_cartan_orb}. For instance, 
\[
\nu(\eta_8)(v)\,=\,4\tau_{F_1}\tau_{F_4}(v)\,=\,\begin{cases} -2(\ep_1^\ast)^2 - 4\ep_1^\ast\ep_2^\ast,  & \text{if } v=F_1\cap F_4;\\
0, & \text{otherwise.}
\end{cases}
\]
One can immediately see that $\nu(\eta_8)$ above  satisfies the condition for $H_{T}^\ast(\Gamma, \alpha)$ in Definition \ref{def_equiv_cohom_of_graph}. 
\end{remark}

\begin{figure}
\begin{tikzpicture}[scale=0.9]

\draw[dotted] (0,0)--(2,0)--(2,2)--(0,2)--cycle;
\node[right] at (0,1) {$\tau_{F_1}$};
\draw[thick] (0,0)--(0,2);
\node[above right] at (-0.2,2) {\scriptsize$-\ep_1^\ast$};
\node[below right] at (-0.2,0) {\scriptsize$-\frac{2}{3}\ep_1^\ast-\frac{1}{3}\ep_2^\ast$};

\draw[fill] (0,2) circle [radius=0.05]; 
\draw[fill] (0,0) circle [radius=0.05];

\begin{scope}[xshift=100]
\draw[dotted] (0,0)--(2,0)--(2,2)--(0,2)--cycle;
\node[below] at (1,0) {$\tau_{F_2}$};

\draw[thick] (0,0)--(2,0);
\node[above] at (0,0) {\scriptsize$-\frac{1}{3}\ep_1^\ast-\frac{2}{3}\ep_2^\ast$};
\node[above] at (2,0) {\scriptsize$-\frac{1}{2}\ep_2^\ast$};

\draw[fill] (2,0) circle [radius=0.05];
\draw[fill] (0,0) circle [radius=0.05];
\end{scope}

\begin{scope}[xshift=200]
\draw[dotted] (0,0)--(2,0)--(2,2)--(0,2)--cycle;
\node[left] at (2,1) {$\tau_{F_3}$};

\draw[thick] (2,0)--(2,2);
\node[above left] at (2.2,2) {\scriptsize$\ep_1^\ast$};
\node[below left] at (2.2,0) {\scriptsize$\ep_1^\ast+\frac{1}{2}\ep_2^\ast$};

\draw[fill] (2,2) circle [radius=0.05];
\draw[fill] (2,0) circle [radius=0.05];
\end{scope}

\begin{scope}[xshift=300]
\draw[dotted] (0,0)--(2,0)--(2,2)--(0,2)--cycle;
\node[above] at (1,2) {$\tau_{F_4}$};

\draw[thick] (0,2)--(2,2);
\node[below] at (0,2) {\scriptsize$\frac{1}{2}\ep_1^\ast+\ep_2^\ast$};
\node[below] at (2,2) {\scriptsize$\ep_2^\ast$};

\draw[fill] (0,2) circle [radius=0.05]; 
\draw[fill] (2,2) circle [radius=0.05];
\end{scope}
\end{tikzpicture}
\caption{Rational Thom classes}
\label{fig_rational_thom_for_cartan_orb}
\end{figure}


\begin{example}[Singular cohomology]\label{ex_toric_orb_cartan_mat_sing_cohom}
Recall that $H^\ast(X(Q, \lambda))\cong \ZZ[\Gamma, \alpha]/\mathcal{J}$ by Corollary \ref{cor_sing_cohom}. Here, $\mathcal{J}$ is the ideal generated by $\{\zeta_1, \zeta_2\}$, 
see Example \ref{ex_toric_orb_cartan_mat}. To obtain the minimal number of generators and relations, we may choose $y\colonequals \zeta_3$ and $z\colonequals \zeta_4$ as generators of $H^2(X(Q, \lambda))$.

Moreover, one can see 
\begin{align}
&\eta_1=\eta_3=\eta_4=\eta_7=0,\label{eq_rel1}\\
&-\eta_5=\eta_8=2\eta_2=2\eta_6. \label{eq_rel2}
\end{align}
modulo two ideals $ \mathcal{I}$ and $\mathcal{J}$. 
To see  \eqref{eq_rel1} and \eqref{eq_rel2} explicitly,  the following relation 
\begin{equation}\label{eq_implicit_relation}
x_{1}x_{2}=2x_{1}^2=2x_{2}^2.
\end{equation}
may be helpful, which can be obtained by the following computation. 
\begin{align*}
x_{2}^2&=(2x_{1}-x_{3})x_{2} & & (\text{by }\zeta_1=-2x_{1}+x_{2}+x_{3}=0) \\
&=2(2x_{2}-x_{4})x_{2}-x_{2}x_{3}&  & (\text{by }\zeta_2=x_{1}-2x_{2}+x_{4}=0)\\
&=4x_{2}^2-x_{2}x_{3}& &(\text{by }x_{2}x_{4}=0)\\
&=4x_{2}^2-x_{2}(2x_{1}-x_{2}) &&(\text{by }\zeta_1=-2x_{1}+x_{2}+x_{3}=0)\\
&=5x_{2}^2-2x_{1}x_{2}, &&
\end{align*}
which implies that $2x_{2}^2=x_{1}x_{2}$. Moreover, 
\begin{align*}
x_{1}x_{2}&=x_{1}(2x_{1}-x_{3}) && (\text{by } \zeta_1=-2x_{1}+x_{2}+x_{3}=0) \\
&=2x_{1}^2. & & (\text{by }x_{1}x_{3}=0)
\end{align*}
Combining these two, we have \eqref{eq_implicit_relation}. Hence, for instance 
\begin{align*}
\eta_1&=x_{1}^2 + 2x_{1}x_{2} +x_{2}^2 + 3x_{3}^2 +x_{4}^2+2x_{1}x_{4}\\
&=12x_{1}^2-10x_{1}x_{2}+8x_{2}^2\\
&=0,
\end{align*}
where the second and third equalities follow from the linear ideal $\mathcal{J}$ and \eqref{eq_implicit_relation}, respectively.  

Thus, we may put $w\colonequals \eta_2=\eta_6$ as  a generator of $H^4(X(Q, \lambda))$. Combining $w$ with degree $2$ generators, we have more relations:
$$y^2=z^2=-2w,~yz=4w,~yw=zw=0$$
modulo the ideal  $ \mathcal{I}+\mathcal{J}$. 
Indeed, these can be verified by the following computation, where we use the ideal $\mathcal{I}+\mathcal{J}$ and \eqref{eq_implicit_relation} again. 
\begin{align*}
y^2&=4x_{3}^2=4(2x_{1}-x_{2})^2=4(4x_{1}^2-4x_{1}x_{2}+x_{2}^2)=-12x_{1}^2,\\
z^2&=4x_{4}^2=4(-x_{1}+2x_{2})^2)=4(x_{1}^2-4x_{1}x_{2}+4x_{2}^2)=-12x_{1}^2,\\
w&=x_{3}x_{4}=(2x_{1}-x_{2})(-x_{1}+2x_{2})=6x_{1}^2, \\
yz&=4x_{3}x_{4}=4w,\\
yw&=2x_{3}(x_{3}x_{4})=2x_{3}\cdot 6x_{1}^2 =0, \\
zw&=2x_{4}(x_{3}x_{4})=2x_{4}\cdot 6x_{2}^2 =0.
\end{align*}

Finally, we conclude that 
$$H^\ast(X(Q, \lambda))\cong \ZZ[y,z,w]/\mathcal{K}$$
where $\deg y=\deg z=2, \deg w=4$ and $\mathcal{K}$ is the ideal generated by 
$$\{y^2+2w, z^2+2w, yz-4w, yw, zw\}.$$

\end{example}
%
%
%

\end{document}